\documentclass[12pt,a4paper]{article}
\usepackage{amsmath,amsfonts,amssymb,amsthm,amscd}
\newtheorem{thm}{Theorem}
\newtheorem{lem}[thm]{Lemma}
\newtheorem{prop}[thm]{Proposition}
\newtheorem{defn}[thm]{Definition}
\newtheorem{cor}[thm]{Corollary}
\newtheorem{rem}[thm]{Remark}

\newtheorem{example}[thm]{Example}

\begin{document}

\author{Liana David and Claus Hertling}

\date{}

\title{Regular  $F$-manifolds: initial conditions and  Frobenius metrics}

\maketitle

\textbf{Abstract:} A regular $F$-manifold is an $F$-manifold (with
Euler field) $(M, \circ , e, E)$, such that the endomorphism
$\mathcal U(X):=E\circ X$ of $TM$  is regular at any $p\in M.$ We
prove that the germ  $((M,p),  \circ ,e, E)$ is uniquely
determined  (up to isomorphism) by the conjugacy class of
${\mathcal U}_{p}: T_{p} M \rightarrow T_{p}M$. We obtain that any
regular $F$-manifold admits a preferred system of local
coordinates and we find conditions, in these coordinates, for a
metric to be Frobenius. We study the Lie algebra of infinitesimal
symmetries of regular $F$-manifolds. We show that any regular
$F$-manifold is locally isomorphic to the parameter space of a
Malgrange universal connection. We prove an initial condition
theorem for Frobenius metrics on regular $F$-manifolds.\\


{\bf 2010 MS Classification}: 32B10, 32G99, 53Z05, 53D45.

\section{Introduction}

Frobenius manifolds were defined in \cite{dubrovin}, by Boris
Dubrovin, as a geometrization of the so called
Witten-Dijkgraaf-Verlinde-Verlinde (WDVV)-equations and appear in
many  areas of mathematics (quantum cohomology, singularity
theory, integrable systems etc). Later on, the weaker notion of
$F$-manifold was introduced in the literature by Hertling and
Manin \cite{hert-man} and was intensively studied since then (see
e.g. \cite{liana,hert-book,lorenzoni,merkulov,strachan}). Rather
than the usual definition of Frobenius manifolds \cite{dubrovin},
we prefer the alternative one \cite{hert-book} where Frobenius
manifolds are viewed as an enrichment of $F$-manifolds.

\begin{defn}\label{definitii} i) An $F$-manifold is a manifold $M$ together with a
(fiber preserving) commutative, associative multiplication $\circ$
on $TM$, with unit field $e$, and an additional field $E$ (called
the Euler field), such that the following conditions hold:
\begin{equation}\label{integr}
L_{X\circ Y} (\circ ) = X\circ L_{Y}(\circ ) + Y\circ L_{X}(\circ
)
\end{equation}
and
$$
L_{E}(\circ )(X, Y)= X\circ Y,
$$
for any vector fields $X, Y\in {\mathcal T}_{M}$.\

ii) A Frobenius manifold is an $F$-manifold $(M ,\circ , e, E)$
together with a (non-degenerate) flat, multiplication invariant
metric $g$ (i.e. $g(X\circ Y, Z) = g(X, Y\circ Z)$, for any $X, Y,
Z\in TM$), such that $L_{E}(g) = D g$ (with $D\in \mathbb{C}$) and
$\nabla^{\mathrm{LC} }(e)=0$ (where $\nabla^{\mathrm{LC}}$ is the
Levi-Civita connection of $g$).\

iii) A Frobenius metric on an $F$-manifold $(M, \circ , e,E)$ is a
metric $g$ which makes $(M, \circ , e, E, g)$ a Frobenius
manifold.

\end{defn}

There are examples of $F$-manifolds which do not support locally
any Frobenius metric (see Remark \ref{non-weak}). In fact, it is
difficult to construct explicitly Frobenius manifolds (one strong
obstruction being the flatness of the metric). The semisimple case
is understood. More precisely, a  semisimple $F$-manifold admits,
by definition,  a coordinate system $(u^{i})$, called canonical,
in which the multiplication takes the simple form
$\partial_{i}\circ
\partial_{j} = \delta_{ij}\partial_{j}$. Any multiplication invariant metric
is diagonal in canonical coordinates, its flatness  is expressed
by the Darboux-Egoroff equations and  Frobenius metrics exist
locally on the open subset $M^{\mathrm{tame}} : = \{ (u^{i}),\
u^{i}\neq u^{j},\ i\neq j\}$ of tamed points (see e.g.
\cite{dubrovin}). More general classes  of Frobenius manifolds can
be obtained from the so called initial condition theorems,
developed by Hertling and Manin in \cite{hert-man-unfol}. It turns
out that the germ $((M,p), \circ , e, E, g)$ of certain Frobenius
manifolds  is determined (modulo isomorphism) by the linear data
induced on the tangent space $T_{p}M$ and conversely, starting
with an abstract linear data (called 'initial condition') one
obtains a unique (up to isomorphism) germ of such Frobenius
manifolds. In the semisimple case, this was already proved in
\cite{dubrovin}, Lecture 3. Generalizations, where the point was
replaced by an entire submanifold, were also developed in
\cite{hert-man-unfol}.

We shall be particularly interested in the relation between
$F$-manifolds and meromorphic connections. The parameter space $M$
of a meromorphic connection $\nabla$ on a vector bundle over
$M\times D$ (where $D\subset \mathbb{C}$ is a small disc around
the origin), with poles of Poincar\'{e} rank one along $M \times
\{ 0\}$, in Birkhoff normal form $( \frac{B_{0}(x)}{\tau } +
B_{\infty})\frac{d\tau}{\tau } + \frac{{\mathcal C}_{i}(x)
dx^{i}}{\tau}$, inherits (under additional conditions), an
$F$-manifold structure. On the other hand, if at a point $x_{0}\in
M$ the matrix $B_{0}(x_{0})$ is regular (see the comments after
Definition \ref{new-def}), then, for a small neighborhood $U$ of
$x_{0}$, $\nabla\vert_{U \times D}$ is uniquely determined (up to
pull-backs $(f\times \mathrm{Id})^{*}$ and isomorphisms) by its
restriction $\nabla^{\mathrm{0}}:= \nabla\vert_{\{ x_{0}\} \times
D}$. (For this reason, $\nabla^{0}$ can be considered as the
'initial condition' for $\nabla$). This follows from the existence
of a universal integrable deformation $\nabla^{\mathrm{can}}$ of a
meromorphic connection (in our case $\nabla^{0}$), on a vector
bundle over $D$, in Birkhoff normal form, with a pole of
Poincar\'{e} rank one in $\{ 0\}$, with regular residue. Such a
universal deformation was constructed by Magrange in
\cite{mal1,mal2}. The parameter space of more general meromorphic
connections (not necessarily in the Birkhoff normal form), the so
called (TE)-structures, is also an $F$-manifold (see
\cite{dutch}).

In this paper we are concerned with a large class of
$F$-manifolds, namely the regular ones, and their relation with
Frobenius metrics and meromorphic connections.

\begin{defn}\label{new-def} An $F$-manifold $(M, \circ , e, E)$ is called  regular if
the endomorphism $\mathcal U : TM \rightarrow TM$,  $\mathcal U (X):=E\circ X$, is regular at any $p\in M.$
\end{defn}

(An endomorphism $A: V \rightarrow V$ of a complex vector space
$V$ is regular if one of the following equivalent conditions
holds: 1) any two distinct Jordan blocks from its Jordan normal
form have distinct eigenvalues; 2) the characteristic and minimal
polynomials of $A$ coincide; 3) the vector space of endomorphisms
of $V$ commuting with $A$ has dimension $n=\mathrm{dim}(V)$ and
basis $\{\mathrm{Id},A, \cdots , A^{n-1}\}$;  4)  there is a
cyclic vector for $A$,  i.e. a vector $v\in V$ such that $\{ v,
A(v), \cdots , A^{n-1}(v)\}$ is a basis of $V$).\

Our main result from this paper is as follows. Its second part can
be understood as an initial condition theorem for regular
$F$-manifolds.

\begin{thm}\label{main-0} i)  Any germ
$((M,p) ,\circ , e, E)$ of regular $F$-manifolds is  isomorphic to
a product ${\mathcal P}:=\Pi_{\alpha
=1}^{n}((\mathbb{C}^{m_{\alpha}},0), \circ_{\alpha} ,e_{\alpha},
E_{\alpha})$ of germs of (regular) $F$-manifolds. Here
$m_{\alpha}$ are the dimensions of the Jordan blocks of the
endomorphism ${\mathcal U}_{p}(X) = X \circ  E_{p}$ of $T_{p}M$.
For each such block, let $a_{\alpha}$ be the corresponding
eigenvalue of ${\mathcal U}_{p}$. In the canonical frame field $\{
\partial_{i}:= \frac{\partial}{\partial t^{i}},\ 0\leq i\leq
m_{\alpha}-1\}$, determined by coordinates $(t^{0}, \cdots ,
t^{m_{\alpha}-1})\in \mathbb{C}^{m_{\alpha}}$, the multiplication
$\circ_{\alpha}$ is given by
\begin{equation}\label{prima-form-1}
\partial_{i}\circ_{\alpha}{\partial}_{j}=
\begin{cases}
{\partial_{i+j}}, \quad i+j \leq m_{\alpha}-1\\
0, \quad i+j\geq m_{\alpha},\\
\end{cases}
\end{equation}
and the unit field and Euler field by
\begin{equation}\label{prima-form-2}
e_{\alpha}= \partial_{0},\quad E_{\alpha} =
(t^{0}+a_{\alpha})\partial_{0} + ( t^{1}+1)
\partial_{1}+ t^{2}\partial_{2}+\cdots + t^{m_{\alpha}-1} \partial_{m_{\alpha}-1} .
\end{equation}
The product $\mathcal P$ is canonically associated to $((M,p)
,\circ , e, E)$ (up to  ordering of its factors) and the
isomorphism between $((M, p),\circ , e , E)$  and $\mathcal P$ is
unique (when such an ordering is fixed).

ii) In particular, there is a unique (up to unique isomorphism)
germ of regular $F$-manifolds $((M,p), \circ , e, E)$, with given
conjugacy class for the endomorphism ${\mathcal U}_{p}(X) := X
\circ E_{p}$ of $T_{p}M$.\end{thm}

(The conjugacy class of an endomorphism is determined by its
Jordan normal form; two endomorphisms, defined on not necessarily
the same vector space, belong to the same conjugacy class if they
can be reduced to the same Jordan normal form).\\

{\bf Structure of the paper.}  The paper is structured as follows.
In Section \ref{prel} we recall, following
\cite{dubrovin,hert,hert-book,mal1,mal2,sabbah}, the basic
definitions and results we need on Frobenius and $F$-manifolds,
Saito bundles and meromorphic connections.

Section \ref{global-sect} represents a first step in the proof of
Theorem \ref{main-0}. Here we prove that a regular $F$-manifold
$(M, \circ , e, E)$ for which the endomorphism  
${\mathcal U}_{p_{0}}= \mathcal C_{E_{p_{0}}}\in \mathrm{End} (T_{p_{0}}M)$
has exactly one eigenvalue (for $p_{0}\in M$ fixed), is globally nilpotent around $p_{0}$
(see Definition \ref{regular-def} and Proposition
\ref{nilp}).

In Section \ref{proof-section} we  prove  Theorem \ref{main-0}.
One can check 'by hand' that each factor
$(\mathbb{C}^{m_{\alpha}},\circ_{\alpha}, e_{\alpha}, E_{\alpha})$
in Theorem \ref{main-0} i) is an $F$-manifold, for which the
multiplication at $0\in \mathbb{C}^{m_{\alpha}}$ by the Euler
field $E_{\alpha}$ is a Jordan block in the canonical frame
$\{\partial_{i}\}$ of $\mathbb{C}^{m_{\alpha}}$, with  eigenvalue
$a_{\alpha}.$ This shows the existence of the germ in Theorem
\ref{main-0} ii). 
In Propositions \ref{conclusion} and \ref{changed} from this section we prove 
the uniqueness    of the germ in Theorem
\ref{main-0} ii). The argument is based on  Hertling's decomposition of
$F$-manifolds \cite{hert-book}, the local classification of $\{ e\}$-structures \cite{sternberg}
and the material from Section
\ref{global-sect}. The uniqueness of
the germ in Theorem \ref{main-0} ii) implies that any germ
$((M,p), \circ , e, E)$ of regular $F$-manifolds is isomorphic to
a product $\mathcal P$, as required in Theorem \ref{main-0} i).
The uniqueness of the isomorphisms in Theorem \ref{main-0} i) and
ii) is a consequence of the fact that any automorphism of a germ
of regular $F$-manifolds is the identity map (see Lemma
\ref{unicitate-coord}).

The next sections  are devoted to applications of Theorem
\ref{main-0}. The local coordinate system on any regular
$F$-manifold $(M, \circ , e, E)$, provided by Theorem \ref{main-0}
i), is similar to the canonical coordinate system on semisimple
$F$-manifolds. In Section \ref{frobenius-sect} we study Frobenius
metrics on $(M, \circ , e, E)$, in these coordinates. We find
conditions for the coidentity $e^{\flat}$ to be closed, the unit
$e$ to be flat and, respectively,  the Euler field $E$ to preserve
the metric (see Proposition \ref{flat-gen}). The picture is
similar to the semisimple case. To express the flatness, we use
Dubrovin's description of Frobenius manifolds (without Euler
fields) with a maximal abelian group of algebraic symmetries
\cite{dubrovin}. We find an alternative formulation for this
description (see Proposition \ref{corolar-general}) and we apply
it in order to obtain the conditions for a multiplication
invariant metric on a regular globally nilpotent $F$-manifold to
be Frobenius (see Theorem \ref{flat-gen}). The conditions are more
involved than in the semisimple case, owing to the generalized
Darboux-Egoroff equations (see Example \ref{2}).

In Section \ref{lie-sim} we define the Lie algebra of
infinitesimal symmetries of a regular $F$-manifold and we compute
it using the coordinate system provided by Theorem \ref{main-0} i)
(see Definition \label{def-inf-dez} and Proposition
\ref{thm-inf-dez}).

In Section \ref{merom-sect} we study the relation between regular
$F$-manifolds and meromorphic connections. As stated above, the
parameter space of certain meromorphic connections are
$F$-manifolds, but the converse is not true (not every
$F$-manifold can be locally obtained in this way, see Remark
\ref{non-weak} b)). We prove that the converse is, however, true,
under the regularity assumption (see Corollary \ref{mer-rel}).
Namely, we determine the $F$-manifold structure of the parameter
spaces $M^{\mathrm{can}}$ of the Malgrange universal deformations
$\nabla^{\mathrm{can}}$, mentioned above, and we show that any
regular $F$-manifold is locally isomorphic to such a parameter
space (see Proposition \ref{model-uni}).

In Section \ref{initial} we prove an initial condition theorem for
Frobenius metrics on regular $F$-manifolds (see Theorem
\ref{question}). This follows from our Theorem \ref{main-0},
combined with Theorem 4.5 of \cite{hert-man-unfol}. While the
arguments from \cite{hert-man-unfol} work in high generality, they
are also quite technical. For completeness of our exposition, we
provide in the appendix (Section \ref{appendix}) an alternative,
simple and self-contained proof for the existence of the extension
of the metric in Theorem \ref{question}, based on our treatment of
regular $F$-manifolds.\\

{\bf Acknowledgements.} L.D. is supported by a Humboldt Research
Fellowship. She thanks University of Mannheim for hospitality and
the Alexander von Humboldt-Stiftung for financial support. Partial
financial support from a CNCS-grant PN-II-ID-PCE-2011-3-0362 is
also acknowledged. C.H. acknowledges support by the DFG-grant He
2287/4-1.

\section{Preliminary material}\label{prel}

This section is intended to fix notation. We work in the
holomorphic category: the manifolds are complex and  the vector
bundles, sections, connections etc are holomorphic.  We denote by
${\mathcal T}_{M}$ the sheaf of holomorphic vector fields on a
complex manifold $M$, by ${\mathcal O}_{M}$ the sheaf of
holomorphic functions on $M$ and by $\Omega^{1}(M, V)$ the sheaf
of holomorphic $1$-forms with values in a vector bundle $V$. In
our conventions, the connection form of a connection $\nabla$ on a
vector bundle $V\rightarrow M$, in a local basis of sections $\{
s_{1}, \cdots , s_{n}\}$ of $V$, is the matrix valued $1$-form
$\Omega = (\Omega_{ij})$, defined by
$\nabla_{X}(s_{i})=\sum_{j=1}^{n}(\Omega_{ji})_{X} s_{j}.$ The
representation of an endomorphism $A\in \mathrm{End}(V)$ (where
$V$ is a vector space), in a basis $\{ v_{1}, \cdots , v_{n}\}$ of
$V$, is the matrix $A= (A_{ij})$, where $A(v_{i}) = \sum_{j=1}^{n}
A_{ji}v_{j}.$

\subsection{$F$-manifolds}

\subsubsection{Hertling's decomposition of $F$-manifolds}

The following  theorem due to Hertling (see \cite{hert-book}, page
16) plays an essential role in the proof of our main result. We
shall use it for regular $F$-manifolds, but we remark that the
regularity condition is not required for its statement in full
generality.

\begin{thm}\label{hertling-thm}
Let $((M, p), \circ ,e, E)$ be a germ of   $F$-manifolds and
$a_{1}, \cdots , a_{n}$ the distinct eigenvalues of the
endomorphism $\mathcal U(X):= X \circ E$ of $TM$ at $p$. Then
$((M, p), \circ ,e, E)$ decomposes into a product
$\Pi_{\alpha=1}^{n}((M_{\alpha},p_{\alpha}),\circ_{\alpha},
e_{\alpha},E_{\alpha})$ of germs of $F$-manifolds. For any $1\leq
\alpha\leq n$, the endomorphism $\mathcal U_{\alpha}(X):=
X\circ_{\alpha} {E_{\alpha}}$ of $TM_{\alpha}$  has precisely one
eigenvalue at $p_{\alpha}$, namely $a_{\alpha}.$

\end{thm}

\subsubsection{Basic facts on globally nilpotent $F$-manifolds}

\begin{defn}\label{regular-def}
An $F$-manifold $(M, \circ , e ,E)$ is called globally nilpotent
if, for any $p\in M$ and $X_{p}\in T_{p}M$,  the endomorphism
${\mathcal C}_{X_{p}} : T_{p}M \rightarrow T_{p}M$, ${\mathcal
C}_{X_{p}}(Y_{p}) := X_{p}\circ Y_{p}$, has exactly one
eigenvalue. Equivalently, if
$$
{\mathcal C}_{X_{p}}= \mu (X_{p}) \mathrm{Id} + N_{X_{p}},
$$
with $\mu (X_{p})\in \mathbb{C}$ and $N_{X_{p}}\in
\mathrm{End}(T_{p}M)$ nilpotent.
\end{defn}

For any globally nilpotent $F$-manifold $(M, \circ , e, E)$ and
$X\in {\mathcal T}_{M}$, the function $\mu (X): M \rightarrow
\mathbb{C}$, $\mu (X)(p):= \mu (X_{p})$ is holomorphic. Indeed,
the characteristic polynomial $P$ of ${\mathcal C}_{X}$ is given
by $P(z, p) = (z-\mu (X_{p}) )^{n}$, for any $z\in \mathbb{C}$ and
$p\in M$ (where $n = \mathrm{dim}(M)$). Therefore, $\mu (X) =
-\frac{1}{n!} P^{(n-1)} (0,\cdot)\in {\mathcal O}_{M}$, as $P$ is
holomorphic (the superscript $(n-1)$ denotes the $(n-1)$
derivatives with respect to $z$). In particular, the eigenfunction
$a := \mu (E)$ of $\mathcal U = {\mathcal C}_{E}$ is holomorphic
(not necessarily constant).

Recall Definition \ref{new-def} of regular $F$-manifolds, from the introduction.

\begin{lem}\label{ajut} Let $(M, \circ ,e ,E)$ be a regular $F$-manifold of
dimension $n$.\

i) For any $k\geq 0$, let $X_{k}:= E\circ \cdots \circ E$
($k$-times), with $X_{0}:=e.$ The vector fields $\{ X_{0}, X_{1},
\cdots , X_{n-1}\}$ are linear independent (at any point) and
\begin{equation}\label{hertling-f}
[X_{i}, X_{j}] = (j-i) X_{i+j-1}, \quad i, j\geq 0.
\end{equation}

ii) Suppose that the endomorphism ${\mathcal U}_{p} (X_{p}) =
X_{p}\circ E_{p}$ of $T_{p}M$ has exactly one eigenvalue, for any
$p\in M.$ Then $(M, \circ , e, E)$ is globally nilpotent.

\end{lem}

\begin{proof}
The linear independence of $\{ X_{0}, \cdots , X_{n-1}\}$ follows
from regularity. Formula (\ref{hertling-f}) was proved in
\cite{hert-man} (and holds on any $F$-manifold, not necessarily
regular). Claim i) follows. We now prove claim ii). By hypothesis,
${\mathcal U} = a \mathrm{Id} + N$ on $TM$, where $a \in {\mathcal
O}_{M}$ and $N: TM \rightarrow TM$ is nilpotent (at any point).
Let $X:= f_{0} X_{0} +\cdots + f_{n-1} X_{n-1}\in {\mathcal
T}_{M}$, where $f_{i}\in {\mathcal O}_{M}$. We obtain
$$
{\mathcal C}_{X}= (\sum_{k=0}^{n-1} f_{k} a^{k} ) \mathrm{Id} +
\sum_{k=1}^{n-1} f_{k} \sum_{p=1}^{k}C_{k}^{p} a^{k-p} N^{p}.
$$
The second term in the right hand side of the above relation is a
nilpotent endomorphism. Our claim follows.\end{proof}

\begin{defn}\label{can-frame} The vector fields $\{ X_{0}, \cdots , X_{n-1}\}$
from Lemma \ref{ajut} i)  form the canonical frame of the regular
$F$-manifold $(M, \circ , e, E).$
\end{defn}

\subsubsection{Frobenius metrics on constant
$F$-manifolds}\label{constant-dub}

The $F$-manifolds and Frobenius manifolds we are interested in
come, by definition, with an Euler field. However, often in the
literature the quasi-homogeneity condition imposed by the Euler
field is considered as an additional obstruction and is not
required in the definition of these structures. This is true for
example in Dubrovin's description of  Frobenius manifolds which
underly constant $F$-manifolds \cite{dubrovin}. Since we need this
description in Subsection \ref{section-flatness}, we recall it
here. We begin with the definition of constant $F$-manifolds.

\begin{defn} An $F$-manifold $(N, \circ , e)$ (without Euler field) is
called constant if it admits a coordinate system (called
canonical) in which the multiplication is constant.
\end{defn}

Such an $F$-manifold has an $n$-dimensional abelian group of
algebraic symmetries, where $n =\mathrm{dim}(N)$ (see
\cite{dubrovin}, page 69). In a canonical coordinate system
$(t^{0}, \cdots , t^{n-1})$, $\partial_{i} \circ
\partial_{j} = c_{ij}^{k}\partial_{k}$, where
$\partial_{i}:= \frac{\partial}{\partial t^{i}}$ and
$c_{ij}^{k}\in \mathbb{C}.$ (Our notation is different from that
used in \cite{dubrovin}: in this reference, $(t^{i})$ denote the
flat coordinates, rather than the canonical ones). We assume that
there is a constant (in canonical coordinates), multiplication
invariant, (non-degenerate) metric on $TN$ and we fix such a
metric $\epsilon$. Using $\epsilon$, we identify $TN$ with
$T^{*}N.$ The multiplication $\circ$ on $TN$ induces a
multiplication, also denoted by $\circ$, on $T^{*}N$. It is given
by: $dt^{i} \circ dt^{j} = c^{ij}_{k} dt^{k}$, where $c^{ij}_{k} =
\epsilon^{is} c_{sk}^{j}.$ A $1$-form $\psi \in \Omega^{1} (N)$ is
called invertible if, for any $p\in N$, the covector $\psi_{p} \in
T^{*}_{p}N$ is invertible with respect to $\circ .$

Let $\gamma\in \mathrm{End}(TN)$ be an $\epsilon$ symmetric
endomorphism  which satisfies the generalized Darboux-Egoroff
equations:
\begin{equation}\label{gen-darboux}
[ {\mathcal C}_{i}, L_{\partial_{j}}(\gamma ) ] - [ {\mathcal
C}_{j}, L_{\partial_{i}}(\gamma ) ] + [ [{\mathcal C}_{i}, \gamma
], [ {\mathcal C}_{j}, \gamma ]] =0,\ 0\leq  i, j\leq n-1,
\end{equation}
where ${\mathcal C}_{i}:= {\mathcal C}_{\partial_{i}}$. Such an
endomorphism  is called, by analogy with the semisimple case, a
rotation coefficient operator. The following description of
Frobenius metrics on constant $F$-manifolds is due to Dubrovin
(see Theorem 3.1 of \cite{dubrovin}). Below $\psi [ {\mathcal
C}_{i}, \gamma ]\in \Omega^{1}(N)$ is the composition (as maps) of
$[ {\mathcal C}_{i}, \gamma ]\in \mathrm{End}(TN)$ with the
$1$-form $\psi : TN \rightarrow \mathbb{C}$).

\begin{thm}\cite{dubrovin}\label{dub-thm}
Let $\gamma$ be a rotation coefficient operator on $(N, \circ ,
e)$ and $\psi\in \Omega^{1} (N)$  invertible, satisfying
\begin{equation}\label{necesitate}
L_{\partial_{i}} (\psi ) = \psi  [ {\mathcal C}_{i}, \gamma ],\
0\leq i\leq n-1.
\end{equation}
Then the metric
\begin{equation}\label{metric-can-def}
g (X, Y):= (\psi\circ \psi ) (X\circ Y),\ X, Y\in TN
\end{equation}
is Frobenius on $(N, \circ , e)$. Conversely, any Frobenius metric
on $(N, \circ , e)$ is of this form, for a certain rotation
coefficient operator $\gamma $ and $1$-form $\psi .$
\end{thm}

\subsection{$F$-manifolds, Frobenius manifolds and Saito bundles}\label{frob-saito}

In this paragraph we recall the relation between Frobenius or
$F$-manifolds and Saito bundles (see e.g. \cite{sabbah}, Chapter
VII).

\begin{defn}\label{saito-def} i)  A Saito bundle is a vector bundle $(\pi : V \rightarrow M,\nabla , \Phi, R_{0},
R_{\infty})$  with a connection $\nabla$,  a  $1$-form $\Phi\in
\Omega^{1} (M, \mathrm{End}(V))$ and two endomorphisms
$R_{0},R_{\infty}\in \mathrm{End}(V)$, such that the following
conditions are satisfied:
$$
R^{\nabla}=0,\quad \Phi\wedge \Phi=0,\quad [R_{0}, \Phi ]=0
$$
and
\begin{equation}\label{saito-rel}
d^{\nabla}\Phi =0,\quad \nabla R_{0} + \Phi = [ \Phi ,
R_{\infty}],\quad \nabla R_{\infty}=0.
\end{equation}
Above $R^{\nabla}$ is the curvature of $\nabla$ and the
$\mathrm{End}(V)$-valued forms $[R, \Phi ]$ (with $R:= R_{0}$ or
$R_{\infty}$), $d^{\nabla} \Phi$ and $\Phi\wedge \Phi$ are defined
by: for any $X, Y\in {\mathcal T}_{M}$,
\begin{align*}
& [R, \Phi ]_{X}:= [ R, \Phi_{X}]\\
& (d^{\nabla}\Phi)_{X,Y}:= \nabla_{X} (\Phi_{Y}) - \nabla_{Y} (\Phi_{X}) - \Phi_{[X,Y]}\\
& (\Phi\wedge \Phi )_{X,Y}:=\Phi_{X} \Phi_{Y} - \Phi_{Y} \Phi_{X}.
\end{align*}

ii) A Saito bundle with metric is a Saito bundle $(V,\nabla ,
\Phi, R_{0},R_{\infty})$ with a (non-degenerate) metric $g\in
S^{2} (V^{*})$, such that the following conditions are satisfied:
$$
\nabla g=0,\quad R_{\infty} + R_{\infty}^{*}=0,\quad R_{0} =
R_{0}^{*},\quad \Phi_{X}= \Phi_{X}^{*},\ \forall X\in TM,
$$
where the superscript "$*$"denotes the $g$-adjoint.
\end{defn}

Let  $(\pi :V\rightarrow M,\nabla , \Phi , R_{0}, R_{\infty})$ be
a Saito bundle. Suppose there is a section $s$ of $V$, such that
\begin{equation}\label{iso-prim}
I: TM\rightarrow V,\quad I(X):=\Phi_{X} (s),\ X\in TM
\end{equation}
is a bundle isomorphism. Define an (associative, commutative, with
unit field $e_{M} := I^{-1}(s)$)  multiplication  $\circ_{M}$ on
$TM$ and a vector field $E_{M}\in {\mathcal T}_{M}$ by the
conditions $\Phi_{X\circ_{M} Y} (s)= \Phi_{X}\Phi_{Y}(s)$ and
$E_{M} = - I^{-1}R_{0}(s) $. (We remark that our conventions
differ from those used in \cite{hert,sabbah}; in these references,
the identification between $TM$ and $V$ is done via $-I$; the
induced multiplication on $TM$ is then $-\circ_{M}$, but the
induced fields $E_{M}$ are the same). The following holds (see
e.g. Lemmas 4.1 and 4.3 of \cite{hert}):

\begin{prop}\label{F-saito}
Both $\circ_{M}$ and $E_{M}$ are independent of the section $s$
and $(M, \circ_{M}, e_{M}, E_{M})$ is an $F$-manifold. The
endomorphism ${\mathcal U}_{M}(X) = X\circ E_{M}$ of $TM$
coincides with $- I^{-1} R_{0} I$. In particular, $({\mathcal
U}_{M})_{p}$ and $-(R_{0})_{p}$ belong to the same conjugacy
class, for any $p\in M.$
\end{prop}

Suppose now that $g$ is a metric on $V$, which makes $(V,\nabla ,
\Phi ,R_{0}, R_{\infty},g)$ a Saito bundle with metric. Suppose
that the section $s$ from the isomorphism (\ref{iso-prim})  is
$\nabla$-parallel and $R_{\infty}(s) = qs$, for $q\in \mathbb{C}$.
Such a section is called primitive homogeneous. Then $g_{M} (X,Y)
:=g( I(X), I(Y))$ is a Frobenius metric on $(M,\circ_{M}, e_{M},
E_{M})$. The Levi-Civita $\nabla^{\mathrm{LC}}$ of $g_{M}$ is
given by $\nabla^{\mathrm{LC}} = I^{-1} \circ \nabla \circ I$ and
\begin{equation}\label{der-Euler}
\nabla^{\mathrm{LC}}E_{M} = I^{-1} R_{\infty}  I +
(1-q)\mathrm{Id}.
\end{equation}
(see \cite{sabbah}, page 239). Conversely, any Frobenius manifold
arises in this way (see \cite{sabbah}, page 240).

\subsection{$F$-manifolds and flat meromorphic
connections}\label{f-merom}

Let $\nabla$ be a flat meromorphic connection on a vector bundle
$E $ over $M\times D$ (where $D\subset\mathbb{C}$ is a small disc
around the origin), with poles of Poincar\'{e} rank one along $M
\times \{ 0\}$, in Birkhoff normal form. By definition, this means
that $E= (M \times D) \times \mathbb{C}^{n}\rightarrow M \times D$
is the trivial bundle and the connection form of $\nabla$ in the
standard trivialization of $E$ is given by
\begin{equation}\label{forma-conex}
\Omega = (\frac{B_{0}(x)}{\tau } + B_{\infty}) \frac{d\tau}{\tau }
+ \frac{{\mathcal C}_{i}(x) dx^{i}}{\tau},
\end{equation}
where $(x^{i})$ are coordinates on $M$, $\tau$ is the coordinate
on $D$, ${\mathcal C}_{i} : M \rightarrow M_{n}(\mathbb{C})$,
$B_{0}: M\rightarrow M_{n}(\mathbb{C})$ and $B_{\infty}\in
M_{n}(\mathbb{C})$. We consider $B_{0}, B_{\infty}$ as
endomorphisms (the latter, constant) of the trivial bundle $V=
M\times \mathbb{C}^{n} \rightarrow M$ and ${\mathcal C}= {\mathcal
C}_{i}dx^{i} $ as an $\mathrm{End}(V)$-valued $1$-form on $M$. In
its simplest form, the relation between $F$-manifolds and
meromorphic connections is the following (for the more general
relation between $F$-manifolds and (TE)-structures, see
\cite{dutch}, Theorem 3.1).

\begin{prop}\label{F-mer} Let $\nabla$ be a flat meromorphic connection on
the trivial bundle $E = (M\times D)\times \mathbb{C}^{n}
\rightarrow M \times D$, in Birkhoff normal form
(\ref{forma-conex}).\

i) The trivial bundle $V = M \times \mathbb{C}^{n}\rightarrow M$,
together with $(D, {\mathcal C}, B_{0},- B_{\infty})$, where $D$
is the canonical flat connection of $V$ and ${\mathcal C}$,
$B_{0}$, $B_{\infty}$ are as above, is a Saito bundle.\

ii) In particular, if there is a section $s$ of $V$, such that
(\ref{iso-prim}) is an isomorphism, then $M$ inherits an
$F$-manifold structure.
\end{prop}

\begin{proof} Claim i) follows from the flatness condition
$d\Omega + \Omega \wedge \Omega =0$. Claim ii) follows from
Proposition \ref{F-saito}.\end{proof}

\subsection{Malgrange universal deformation}\label{malgrange-sect}

Let $\nabla^{0}$ be a connection on the trivial bundle $V^{0}=
D\times \mathbb{C}^{n} \rightarrow D$, with connection form
\begin{equation}\label{omega-0}
\Omega^{0}=(\frac{B_{0}^{o}}{\tau } + B_{\infty})
\frac{d\tau}{\tau },
\end{equation}
where $B^{o}_{0}, B_{\infty}\in M_{n}(\mathbb{C})$. To keep the
text self-contained, we recall the definition of an integrable
deformation of $\nabla^{0}$.

\begin{defn}
An integrable deformation of $\nabla^{0}$ is a flat meromorphic
connection $\nabla$ on the trivial vector bundle $E = (M \times
D)\times \mathbb{C}^{n} \rightarrow M \times D$, in  Birkhoff
normal form (\ref{forma-conex}),  which coincides with
$\nabla^{0}$ when restricted to $\{ p_{0}\} \times D$ (where
$p_{0}\in M$).
\end{defn}

Assume now that $B^{o}_{0}$ is  regular. Then $\nabla^{0}$ admits
an integrable deformation $\nabla^{\mathrm{can}}$, constructed by
Malgrange \cite{mal1,mal2}, which is universal (see e.g.
\cite{sabbah}, page 208, for the definition of universal
integrable deformations). Following Sabbah (\cite{sabbah}, Chapter
VI, Section 3.a), we now recall its construction. Let $\mathcal
D\subset T (M_{n}(\mathbb{C}^{n}))$ be defined by
\begin{equation}\label{rond}
{\mathcal D}_{\Gamma}:= \mathrm{Span}_{\mathbb{C}} \{ \mathrm{Id},
(B_{0})_{\Gamma}, \cdots ,(B_{0})_{\Gamma }^{n-1}\}\subset
T_{\Gamma}M_{n}({\mathbb{C}}) = M_{n}({\mathbb{C}}),
\end{equation}
where
\begin{equation}\label{B-0}
(B_{0})_{\Gamma }:= B^{o}_{0}-\Gamma  +[B_{\infty}, \Gamma ].
\end{equation}
Because $B^{o}_{0}$ is regular, so is $(B_{0})_{\Gamma}$, for any
$\Gamma\in W$, where $W$ is a small open neighborhood of  $0$ in
$M_{n}({\mathbb{C}})$. For any $\Gamma \in W$,  $\mathcal
D_{\Gamma}$ is the ($n$-dimensional) vector space of polynomials
in $(B_{0})_{\Gamma}$  and the distribution  $\mathcal
D\rightarrow W$ is integrable. The Malgrange universal deformation
of $\nabla^{0}$ is defined as follows \cite{mal1,mal2} (also
\cite{sabbah}):

\begin{defn}\label{def-para} i) The parameter space $M^{\mathrm{can}}= M^{\mathrm{can}} (B_{0}^{o}, B_{\infty})$
of the universal deformation $\nabla^{\mathrm{can}}$ of
$\nabla^{0}$ is the maximal integral submanifold of $\mathcal
D\vert_{W}$, passing through $0$.

ii) The connection $\nabla^{\mathrm{can}}$ of $\nabla^{0}$ is
defined on the trivial bundle $E = (M^{\mathrm{can}} \times
D)\times \mathbb{C}^{n} \rightarrow M^{\mathrm{can}}\times D$,
with connection form in the standard trivialization of $E$ given
by
\begin{equation}\label{omega-cann}
\Omega^{\mathrm{can}} = \left( \frac{B_{0}}{\tau} +
B_{\infty}\right) \frac{d\tau}{\tau } +\frac{\mathcal C}{\tau}.
\end{equation}
Here $B_{0} : M^{\mathrm{can}}\rightarrow M_{n}(\mathbb{C})$,
$(B_{0})(\Gamma) : = (B_{0})_{\Gamma}$ is given by (\ref{B-0}) and
${\mathcal C}_{X}:= X$ is the action of the matrix $X$ on
$\mathbb{C}^{n}$, for any $X\in T_{\Gamma}M^{\mathrm{can}}\subset
M_{n}(\mathbb{C})$.

\end{defn}

\section{Globally nilpotent regular
$F$-manifolds}\label{global-sect}

Our aim in this section is to prove the following result.

\begin{prop}\label{nilp}Let $(M, \circ , e, E)$ be an $F$-manifold
of dimension $n\geq 2$, such that at a point $p_{0}\in M$, the
endomorphism ${\mathcal U}_{p_{0}} = ({\mathcal C}_{E})_{p_{0}} :
T_{p_{0}}M \rightarrow T_{p_{0}} M$ is regular, with exactly one
eigenvalue. Then there is a neighborhood $U$ of $p_{0}$, such that
$(U, \circ ,e , E)$ is globally nilpotent (and regular).
\end{prop}

\begin{proof}
Let $U$ be  a small neighborhood of $p_{0}$, such that, for any
$p\in U$, the endomorphism ${\mathcal U}_{p}= ({\mathcal
C}_{E})_{p}: T_{p}M \rightarrow T_{p}M$  is regular, and let
$P(p,z) = z^{n} + \sum_{k=0}^{n-1} \lambda_{k} (p) z^{k}$ be  the
characteristic (or minimal) polynomial  of ${\mathcal U}_{p}$.
Denote by $\{ X_{0} ,X_{1}, \cdots ,\ X_{n-1}\}$  the canonical
frame of the regular $F$-manifold $(U, \circ ,e, E)$ (see
Definition \ref{can-frame}). Since $P(p,{\mathcal U}_{p})=0$,
\begin{equation}\label{f}
X_{n} + \sum_{k=0}^{n-1} \lambda_{k} X_{k} =0.
\end{equation}
Define the functions
$$
f_{k}:= \lambda_{k} - \frac{ C_{n}^{k}}{ n^{n-k}}
\lambda_{n-1}^{n-k},\ 0\leq k\leq n-2.
$$
(Remark that $f_{k}(p)=0$ for $p\in U$  and all $0\leq k\leq n-2$,
if and only if  $P(p, z) = (z + \frac{\lambda_{n-1}(p)}{n})^{n}$,
if and only if $\mathcal U_{p}$ has exactly one eigenvalue). By
hypothesis, $f_{k} (p_{0})=0$, for any $ 0\leq k\leq n-2.$ Our aim
is to compute the derivatives $X_{i}(f_{k})$ and to show, using
the Cauchy-Kovalevskaia theorem, that $f_{k}=0$ on $U$, for any
$0\leq k\leq n-2.$ This implies that $\mathcal U$ has exactly one
eigenvalue at any point of $U$ and we conclude from Lemma
\ref{ajut} that $(U, \circ , e, E)$ is globally nilpotent, as
required. Details are as follows.

We take the Lie
derivative of (\ref{f}) with respect to $X_{s}$ ($s\geq 0$) and we use
(\ref{hertling-f}). We obtain
\begin{equation}
(n-s) X_{s+n-1} + \sum_{k=0}^{n-1} \left( X_{s} (\lambda_{k})X_{k} +
(k-s) \lambda_{k} X_{s+k-1}\right) =0,
\end{equation}
which is equivalent, by taking $\lambda_{n}:=1$, to
\begin{equation}\label{ajutatoare}
\sum_{k=0}^{n-1} X_{s} (\lambda_{k})X_{k} +\sum_{k=0}^{n}(k-s) \lambda_{k} X_{s+k-1}=0.
\end{equation}
Relation (\ref{ajutatoare}), with $s=0$, gives
\begin{equation}\label{0}
X_{0} (\lambda_{k}) =- (k+1) \lambda_{k+1},\ 0\leq k \leq n-1.
\end{equation}
Relation (\ref{ajutatoare}), with $s=1$, gives
$$
\sum_{k=0}^{n-1} X_{1} (\lambda_{k}) X_{k} +\sum_{k=0}^{n} (k-1)\lambda_{k} X_{k}=0.
$$
From (\ref{f}),   $X_{n}= - \sum_{k=0}^{n-1} \lambda_{k} X_{k}$.
Replacing this expression of $X_{n}$ into the above relation we obtain
\begin{equation}\label{1}
X_{1} (\lambda_{k})=  (n-k) \lambda_{k}, \quad  0\leq k\leq n-1.
\end{equation}
The computation of $X_{2}(\lambda_{k})$ is done in the same way, but is a bit more complicated.
Taking in
(\ref{ajutatoare}) $s=2$  we obtain
$$
\sum_{k=0}^{n-1} X_{2} (\lambda_{k}) X_{k} + \sum_{k=0}^{n}
(k-2)\lambda_{k} X_{k+1} =0
$$
or, equivalently,
$$
\sum_{k=0}^{n-1} X_{2} (\lambda_{k}) {\mathcal U}^{k} +
\sum_{k=0}^{n} (k-2)\lambda_{k} {\mathcal U}^{k+1} =0.
$$
Since $P(p,\cdot )$ is the minimal polynomial of $\mathcal U_{p}$
(for any $p\in U$) there are holomorphic functions $b_{0},b_{1}\in
{\mathcal O}_{U}$ such that
\begin{equation}\label{ajut-2}
\sum_{k=0}^{n-1} X_{2}(\lambda_{k}) z^{k} + \sum_{k=0}^{n}
(k-2)\lambda_{k} z^{k+1}  = (b_{0} + b_{1} z)
\sum_{k=0}^{n} \lambda_{k}z^{k}.
\end{equation}
Identifying in (\ref{ajut-2})  the coefficients of $z^{n+1}$ we
obtain $b_{1} = n-2$. Relation (\ref{ajut-2}) becomes
\begin{equation}\label{ajut-2-2}
\sum_{k=0}^{n} (X_{2} (\lambda_{k}) - b_{0} \lambda_{k})z^{k} -\sum_{k=1}^{n}
( n-k+1) \lambda_{k-1} z^{k}=0.
\end{equation}
Then (from the coefficient of $z^{0}$),  $X_{2} (\lambda_{0}) = b_{0} \lambda_{0}$, and the remaining
terms in (\ref{ajut-2-2})  give
\begin{equation}\label{aj-2}
\sum_{k=1}^{n} (X_{2} (\lambda_{k}) - b_{0} \lambda_{k} -
(n-k+1) \lambda_{k-1}) z^{k}=0.
\end{equation}
Using  $\lambda_{n}=1$,
we obtain  $b_{0} =-\lambda_{n-1}$ (from the coefficient of $z^{n}$) and
relation (\ref{aj-2}) becomes
\begin{equation}\label{2-prime}
X_{2}
(\lambda_{k}) = - \lambda_{n-1} \lambda_{k} + (n-k+1)\lambda_{k-1},\
0\leq k\leq n-1.
\end{equation}
(We use the convention $\lambda_{i}=0$ for $i <0$; similarly,
below $f_{i}=0$ whenever $i<0$).

Next, we compute the derivatives $X_{3} (\lambda_{k})$.
Relation (\ref{ajutatoare}), with $s=3$, gives, by a similar argument as for $s=2$,
\begin{equation}\label{3}
X_{3} (\lambda_{k})=  ( \lambda^{2}_{n-1} - 2\lambda_{n-2})
\lambda_{k} + (n-k +2) \lambda_{k-2} - \lambda_{n-1}
\lambda_{k-1},\ 0\leq k\leq n-1.
\end{equation}

From relations (\ref{0}), (\ref{1}), (\ref{2-prime})  and
(\ref{3}), we obtain, from long but straightforward computations,
the expressions for the derivatives $X_{i}(f_{k})$ (for any $0\leq
i\leq 3$ and  $0\leq k\leq n-2$):
\begin{equation}\label{f-1}
X_{0}(f_{k}) =- (k+1) f_{k+1}\ (k\leq n-3),\ X_{0} (f_{n-2})
=0,\quad X_{1}(f_{k}) = (n-k) f_{k}
\end{equation}
and
\begin{align}
\nonumber&X_{2} (f_{k})= -\lambda_{n-1} f_{k} + (n-k+1) f_{k-1} -\frac{2 C_{n}^{k}(n-k)}{n^{n-k}}\lambda_{n-1}^{n-k-1}f_{n-2},\\
\nonumber&X_{3}(f_{k})= (\lambda^{2}_{n-1} -2 \lambda_{n-2})f_{k} +  \frac{ C_{n}^{k} (3n-3k-2) }{n^{n-k}}
\lambda_{n-1}^{n-k}
f_{n-2}\\
\label{f-2}& + (n-k+2) f_{k-2} -\lambda_{n-1}f_{k-1} -\frac{3C_{n}^{k}(n-k)}{n^{n-k}}\lambda_{n-1}^{n-k-1}f_{n-3}.
\end{align}
In particular,  both (\ref{f-1}) and  (\ref{f-2}) are of the form
\begin{equation}\label{final}
X_{i} (f_{k}) = \sum_{s=0}^{n-2} a^{(i)}_{ks} f_{s},\ 0\leq i\leq
3,\ 0\leq k\leq n-2,
\end{equation}
for some  $a_{ks}^{(i)}\in {\mathcal O}_{U}$ . Since $[X_{i},
X_{j}]= (j-i) X_{i+j-1}$ (see relation (\ref{hertling-f})),  we
obtain that the derivatives $X_{i}(f_{k})$, for any $0\leq i\leq
n-1$ (and $0\leq k\leq n-2$), are of the form (\ref{final}). In a
coordinate  chart $\chi = (y^{0},\cdots , y^{n-1}): U \rightarrow
\mathbb{C}^{n}$ with $\chi (p_{0}) =0$  we obtain
$$
\frac{\partial (f_{k}\circ \chi^{-1})}{\partial y^{i}} =
\sum_{s=0}^{n-2} b^{(i)}_{ks} (f_{s}\circ \chi^{-1}), \quad 0\leq
i\leq n-1,\quad 0\leq k \leq n-2,
$$
for some  $b^{(i)}_{ks}\in {\mathcal O}_{\chi (U)}$. Also,
$(f_{k}\circ \chi^{-1})(0)=0$ for any $0\leq k\leq n-2$. Applying
successively the uniqueness statement of the Cauchy-Kovalevskaia
theorem (in the form stated e.g. in \cite{hert-man-unfol},
relations (2.42) and (2.43), with no $(t^{i})$-parameters in the
notation of this reference), we obtain that $f_{k}=0$ on $U$, as
required.\end{proof}

The computations from the above proof imply the following
corollary.

\begin{cor}\label{lie-nilp}
Let $(M, \circ , e, E)$ be an $n$-dimensional globally nilpotent
regular $F$-manifold and $\{ X_{0}, \cdots , X_{n-1}\}$ its
canonical frame. Let $a\in {\mathcal O}_{M}$ be the eigenfunction
of  ${\mathcal U}= {\mathcal C}_{E}$. Then
\begin{equation}\label{br}
[X_{i}, X_{j}] = \begin{cases}
(j-i) X_{i+j-1}, \quad i+j \leq n\\
(i-j) \sum_{k=0}^{n-1} c_{k}^{(i+j-1-n)} a^{i+j-1-k} X_{k}, \quad i+j>n,\\
\end{cases}
\end{equation}
where $c_{k}^{(p)}$ ($p\geq 0$ and $0\leq k \leq n-1$)
are constants, defined inductively by $c_{k}^{(0)}
= (-1)^{n-k}C_{n}^{k}$ and  for any $s\geq 0$,
$c_{k}^{(s+1)} = c_{k-1}^{(s)} - c_{k}^{(0)}
c_{n-1}^{(s)}$ (when $k\geq 1$)  and $c_{0}^{(s+1)} = - c_{0}^{(0)} c_{n-1}^{(s)}$.
Moreover,
\begin{equation}\label{deriv}
 X_{i}(a) = a^{i}, \quad i\geq 0.
\end{equation}

\end{cor}

\begin{proof}
Relation (\ref{br}) for $i+j\leq n$ is just (\ref{hertling-f}). We now prove (\ref{br}) for
$i+j > n$.
Since $({\mathcal U} - a\mathrm{Id})^{n}=0$,
\begin{equation}\label{u}
{\mathcal U}^{n}+ \sum_{k=0}^{n-1}c_{k}^{(0)} a^{n-k} {\mathcal U}^{k}=0.
\end{equation}
Multiplying the above relation with $\mathcal U$, $\mathcal U^{2}$, etc, and using
an induction argument, we obtain
$$
{\mathcal U}^{n+s} + \sum_{k=0}^{n-1} c_{k}^{(s)} a^{n-k+s}{\mathcal U}^{k}=0,\quad s\geq 0,
$$
or, equivalently,
\begin{equation}\label{x}
X_{n+s}  = -  \sum_{k=0}^{n-1} c_{k}^{(s)} a^{n-k+s}X_{k},\quad s\geq 0.
\end{equation}
Relations (\ref{hertling-f}) and (\ref{x}) imply (\ref{br}) for
$i+j>n$, as required.

It remains to prove (\ref{deriv}). With the notation from the
proof of Proposition \ref{nilp}, $a=- \frac{\lambda_{n-1}}{n}$ and
$\lambda_{k}= C_{n}^{k} (-a)^{n-k}$, for any $0\leq k\leq n-1$
(because $f_{k}=0$, $(M, \circ , e, E)$ being globally nilpotent).
On the other hand, in the proof of Proposition \ref{nilp} we
computed the following derivatives:
\begin{align*}
&X_{0}(\lambda_{n-1} )= -n,\ X_{1} (\lambda_{n-1}) = \lambda_{n-1}\\
&X_{2}(\lambda_{n-1})=-\lambda_{n-1}^{2}+2\lambda_{n-2}\\
&X_{3} (\lambda_{n-1})= \lambda_{n-1}^{3} + 3 \lambda_{n-3} - 3\lambda_{n-1}\lambda_{n-2}.
\end{align*}
These expressions, written in terms of $a$, give  (\ref{deriv}),
for $0\leq i\leq 3.$ Using (\ref{hertling-f})  we obtain
(\ref{deriv}), for any $i\geq 0$.
\end{proof}

\begin{rem}\label{notatie}{\rm  For the proof of Theorem \ref{main-0} (next section), it is  convenient to express the Lie
brackets $[X_{i}, X_{j}]$, computed in Corollary \ref{lie-nilp},
in a unified form (not as in (\ref{br}), where the cases $i+j\leq
n$ and $i+j>n$ are separated). This can be done as follows.
Consider the constants $c_{k}^{(p)}$ from Corollary
\ref{lie-nilp}. They were defined for  $p\geq 0$ and $0\leq k\leq
n-1$.  For $p<0$ (and $0\leq k\leq n-1$),  let $c_{k}^{(p)}:=0$,
unless $p=k-n$, in which case $c_{k}^{(k-n)}:=-1$. With this
notation, the two relations (\ref{br}) reduce to the single one
$$
[X_{i}, X_{j}] = (i-j) \sum_{k=0}^{n-1}
c_{k}^{(i+j-1-n)}a^{i+j-1-k} X_{k}, \quad i, j\geq 0.
 $$}
\end{rem}

\section{Proof of Theorem \ref{main-0}}\label{proof-section}

Using the material from the previous section, we now prove Theorem
\ref{main-0}. With the explanations from the introduction, the
only statements which  need to be proved are the uniqueness  (up to isomorphism) of the germ in Theorem \ref{main-0} ii) and the uniqueness  
of the isomorphisms in
Theorem \ref{main-0} i) and ii).

We begin by proving the uniqueness of the germ.
Consider two germs $((M, p),
\circ_{M}, e_{M}, E_{M})$ and $((N, q),\circ_{N}, e_{N},E_{N} )$
of  $n$-dimensional $F$-manifolds. Let  ${\mathcal
U}_{M}\in \mathrm{End}(TM)$ and ${\mathcal U}_{N}\in \mathrm{End}(TN)$  be the endomorphisms given by the multiplication with the Euler fields.
We assume that $({\mathcal U}_{M})_{p} : T_{p}M
\rightarrow T_{p}M$ and $({\mathcal U}_{N})_{q} : T_{q} N
\rightarrow T_{q} N$ are regular and belong to the same conjugacy
class. Our  aim is to show that  the two germs are isomorphic. 
Owing to
Hertling's decomposition of $F$-manifolds (see Theorem
\ref{hertling-thm}), we can (and will) assume that $(\mathcal U_{M})_{p}$ and $(\mathcal U_{N})_{q}$ 
have exactly one eigenvalue. From Proposition \ref{nilp}, the germs
$((M, p),\circ_{M} , e_{M}, E_{M})$ and
$((N,q),\circ_{N}, e_{N}, E_{N})$ are globally
nilpotent.  Let $a\in {\mathcal O}_{M}$ and $b\in {\mathcal O}_{N}$ 
be the eigenfunction of $\mathcal U_{M}$ and $\mathcal U_{N}$,  respectively. 
Since $(\mathcal U_{M})_{p}$ and $(\mathcal U_{N})_{q}$ belong to the same 
conjugacy class, $a(p) = b(q).$  We denote by  $\{ X_{i} := (E_{M})^{i}, 0\leq i\leq n-1 \}$ and 
$\{ Y_{i}:= (E_{N})^{i},\ 0\leq i\leq n-1\}$  the canonical frames
of the two germs.

\begin{prop}\label{conclusion} 
In the above setting, there is a biholomorphic transformation $\psi : (M,p) \rightarrow (N,q)$, 
such that  
$\psi_{*} (X_{i}) = Y_{i}$ for any $0\leq i\leq n-1$ and $b\circ \psi = a.$
\end{prop}

\begin{proof} The statement follows from the classification of $\{ e\}$-structures (see 
Theorem 4.1 of \cite{sternberg}, page 344).
For completeness of our exposition we present  the argument in detail, by adapting  the proof of Theorem 4.1 
of \cite{sternberg} to our setting. 
Let $\{ \omega_{i}^{X}\,\ 0\leq i\leq n-1\}$ and $\{\omega_{i}^{Y},\ 0\leq i\leq n-1 \}$ be the $1$-forms dual 
to $\{ X_{i},\ 0\leq i\leq n-1\}$  and $\{ Y_{i},\ 0\leq i\leq n-1\}$, respectively. 
From Remark \ref{notatie},  
\begin{equation}\label{lie-form}
d\omega^{X}_{k} =\sum_{0\leq i,j\leq n-1} (c^{X})_{ij}^{k} \omega^{X}_{i}\wedge \omega^{X}_{j},\quad 
(c^{X})_{ij}^{k} := -\frac{1}{2}(i-j) c_{k}^{(i+j-1-n)} a^{i+j-1-k}
\end{equation}
and similarly 
\begin{equation}\label{lie-form-1}
d\omega^{Y}_{k} =\sum_{0\leq i,j\leq n-1} (c^{Y})_{ij}^{k} \omega^{Y}_{i}\wedge \omega^{Y}_{j},\quad 
(c^{Y})_{ij}^{k} := -\frac{1}{2}(i-j) c_{k}^{(i+j-1-n)} b^{i+j-1-k}.
\end{equation}
From relation (\ref{deriv}),  
\begin{equation}\label{dab}
da = \sum_{i=0}^{n-1}a^{i} \omega_{i}^{X},\quad db = \sum_{i=0}^{n-1}b^{i} \omega_{i}^{Y}.
\end{equation}
In particular, $d_{p}a\in T_{p}^{*}M$ and $d_{q}b\in T_{q}^{*}N$ are non-trivial. We restrict $M$ and $N$ such that $d_{x} a\neq 0$ and $d_{y}b\neq 0$, for any
$x\in M$ and $y\in N.$ Let $\pi_{1}: M\times N \rightarrow M$ and $\pi_{2} : M\times N \rightarrow N$ be the natural projections, $\theta_{i}^{X}:= (\pi_{1})^{*} (\omega_{i}^{X})$
and $\theta_{i}^{Y}:= (\pi_{2})^{*} (\omega_{i}^{Y})$ for any $i$. Pulling back (\ref{dab}) to $M\times N$, we obtain
\begin{equation}\label{dab-1}
d (\pi_{1}^{*} a) = \sum_{i=0}^{n-1}(\pi_{1}^{*}a)^{i} \theta_{i}^{X},\quad d(\pi_{2}^{*}b) = \sum_{i=0}^{n-1}(\pi_{2}^{*}b)^{i} \theta_{i}^{Y}.
\end{equation}
Let $S$  be the $(2n-1)$-dimensional submanifold of $M\times N$, defined by
$$
S:= \{ (x,y)\in M\times N,\ a(x) = b(y) \} .
$$ 
Remark that $(p,q)\in S.$  Let $\vec{i}: S\rightarrow M\times N$ be the inclusion. We will show that 
$$
{\mathcal D}_{S}:= \mathrm{Span}\{ (\vec{i})^{*}( \theta_{i}^{X} - \theta_{i}^{Y}),\ 0\leq i\leq n-1\}\subset T^{*}S
$$
is a rank $(n-1)$ distribution on $S$. For this, we remark from  (\ref{dab-1}) that  
\begin{equation}\label{ts}
\theta_{0}^{X} - \theta_{0}^{Y} = d (\pi_{1}^{*} a - \pi_{2}^{*}b) -  \sum_{i=1}^{n-1}((\pi_{1}^{*}a)^{i} \theta_{i}^{X} - (\pi_{2}^{*}b)^{i} \theta_{i}^{Y}) .
\end{equation}
Restricting (\ref{ts}) to $TS$ and using that  $\pi_{1}^{*}a=\pi_{2}^{*}b$ on 
$S$ (and $d(\pi_{1}^{*} a -\pi_{2}^{*}b) =0$ on $TS$), 
we obtain that $(\vec{i})^{*} (\theta_{0}^{X} - \theta_{0}^{Y})$ is a linear combination of the remaining 
$(\vec{i})^{*} (\theta_{i}^{X} - \theta_{i}^{Y})$, $1\leq i\leq n-1.$ We deduce that 
$\mathrm{rank}( \mathcal D_{S}) \leq n-1.$ On the other hand, since $S$ is of codimension 
$1$ in $M\times N$, the kernel of  the map $(\vec{i})^{*}:T^{*}_{(x,y)}(M\times N)\rightarrow T^{*}_{(x,y)}S$ 
is one dimensional (for any $(x,y)\in S$). We deduce that  
the kernel of  the restriction of this map to $\mathrm{Span} \{ (\theta_{i}^{X} -\theta_{i}^{Y})_{(x,y)},\ 
0\leq i\leq n-1 \}\subset T^{*}_{(x,y)} (M\times N)$ is at most one dimensional.  
Since  $\{ (\theta_{i}^{X}- \theta_{i}^{Y})_{(x,y)}\} $ are linearly independent,  
we obtain that
$\mathrm{dim } (\mathcal D_{S})_{(x,y)} \geq  n-1.$
We conclude that $\mathrm{dim} (\mathcal D_{S})_{(x,y)} = n-1$, for any $(x,y)\in S$,
i.e. $\mathcal D_{S}$ is of rank $n-1$, as needed.  

We now prove that the distribution $\mathcal D_{S}$ is involutive.
Pulling back  the first relations (\ref{lie-form}) and 
(\ref{lie-form-1}) to $S$, 
and using $\pi_{1}^{*}(c^{X})_{ij}^{k} =\pi_{2}^{*} (c^{Y})_{ij}^{k}$ on $S$
(which follows from the definition of $(c^{X})_{ij}^{k}$ and $(c^{Y})_{ij}^{k}$  
and from $\pi_{1}^{*}a=\pi_{2}^{*}b$ on $S$), 
we obtain that 
$$
d(\theta_{k}^{X} - \theta_{k}^{Y}) = \sum_{i,j}\pi_{1}^{*}  (c^{X})^{k}_{ij} (\theta_{i}^{X} - \theta_{i}^{Y})\wedge 
\theta^{X}_{j} + \sum_{i,j}\pi_{2}^{*} (c^{X})^{k}_{ij} \theta^{Y}_{i} \wedge
(\theta^{X}_{j}-\theta^{Y}_{j}),\quad \forall k
$$
i.e.  $\mathcal D_{S}$ is an integrable distribution on $S$.
 
Consider now the integral submanifold  $S^{\prime} $ of  $\mathcal D_{S}$ which contains $(p,q)$.
It is of dimension $\mathrm{dim}(S) - \mathrm{rank} (\mathcal D_{S}) = (2n-1) - (n-1) = n$. 
We claim that 
$\{ (\theta_{i}^{X})\vert_{TS^{\prime}},\ 0\leq i\leq n-1\}$ is a basis of forms
on $S^{\prime}.$ 
Indeed,  
since the forms $\{ (\theta_{i}^{X})_{(x,y)}, (\theta_{i}^{X} - \theta_{i}^{Y})_{(x,y)},\ 0\leq i\leq n-1 \}$ 
are a  basis of $T^{*}_{(x,y)}(M\times N)$, their restriction to  $T_{(x,y)}S^{\prime}$
generate $T_{(x,y)}^{*}S^{\prime}$. Therefore,   
$\{ (\theta_{i}^{X})\vert_{T_{(x,y)}S^{\prime}},\ 0\leq i\leq n-1\}$ 
generate $T_{(x,y)}^{*}S^{\prime}$, hence form  a basis of $T_{(x,y)}^{*}S^{\prime}$ (because
$(\theta_{i}^{X} - \theta_{i}^{Y})_{(x,y)}$ vanishes on $T_{(x,y)}S^{\prime}=( {\mathcal D}_{S})_{(x,y)}$ and $\mathrm{dim}(S^{\prime}) =n$).
We proved that $\{ (\theta_{i}^{X})\vert_{TS^{\prime}},\ 0\leq i\leq n-1\}$ is a basis of forms
on $S^{\prime}$, as needed. 
From this fact and
$(\pi_{1}\vert_{S^{\prime}})^{*} (\omega_{i}^{X}) = (\theta_{i}^{X})\vert_{TS^{\prime}}$, we obtain
that $\pi_{1}\vert_{S^{\prime}} : S^{\prime} \rightarrow M$ is locally a biholomorphic 
transformation. 
A similar  argument shows that $\pi_{2}\vert_{S^{\prime}}: S^{\prime} \rightarrow N$ is also, locally,  a 
biholomorphic transformation.
We restrict the representatives $M$ and $N$ of the germs, such that $\pi_{1} : S^{\prime} \rightarrow M$ and $\pi_{2}: S^{\prime} \rightarrow N$ are biholomorphic transformations
and we define $\psi : =\pi_{2}\circ \pi_{1}^{-1}.$
Since $(p,q)\in S^{\prime}$, $\psi (p) =q.$
Since any $(x,y)\in S^{\prime}$ satisfies $a(x) = b(y)$, we obtain that $b\circ \psi = a.$ Since $\pi_{1}^{*} (\omega_{i}^{X}) = \pi_{2}^{*} (\omega_{i}^{Y})$ on $TS^{\prime}$, 
we obtain that  $\psi^{*} ( \omega_{i}^{Y}) = \omega_{i}^{X}$, i.e.  $\psi_{*} (X_{i}) = Y_{i}$, for any $0\leq i\leq n-1$.
\end{proof}

\begin{prop}\label{changed}
The map $\psi : ((M, p),\circ_{M}, e_{M}, E_{M}) \rightarrow  ((N, q),\circ_{N}, e_{N}, E_{N})$ from Proposition \ref{conclusion} is an isomorphism of germs of $F$-manifolds. 
\end{prop}

\begin{proof}
From Proposition \ref{conclusion}
the map $\psi$ preserves the unit and Euler fields.   It remains to check that it preserves the multiplications, too.
From regularity, this is equivalent to $\psi_{*} (E_{M}^{i}) = E_{N}^{i}$, for any $i\geq 0.$ The statement for $i\leq n-1$ follows 
from Proposition \ref{conclusion} (since $E_{M}^{i}= X_{i}$ and $E_{N}^{i}= Y_{i}$ for such $i$). We need to prove that $\psi_{*} (E_{M}^{i}) = E_{N}^{i}$  also for
$i\geq n.$ For this, we notice that  the characteristic polynomials
of $(\mathcal U_{M})_{x}$ and $({\mathcal U}_{N})_{\psi (x)}$
coincide, for any $x\in M$  (both $(\mathcal U_{M})_{x}$ and $({\mathcal
U}_{N})_{\psi (x)}$ are regular, defined on vector spaces of the
same dimension, with the same (unique) eigenvalue $a(x)=(b\circ
\psi )(x)$). Therefore, for any $i\geq n$, the coordinates of $(E_{M})_{x}^{i}$ in the basis
$\{ (X_{j})_{x},\ 0\leq j\leq n-1\}$  
coincide with the coordinates of $(E_{N})_{\psi (x)}^{i}$ in the basis
$\{ (Y_{j})_{\psi (x)},\  0\leq j\leq n-1\} .$  
Using that $\psi_{*} (X_{j}) = Y_{j}$ 
we deduce that $\psi_{*} ( (E_{M}^{i})_{x}) = (E_{N}^{i})_{\psi (x)}$, i.e.
$\psi_{*} (E^{i}_{M}) = E_{N}^{i}$, as needed.
\end{proof}

The uniqueness of the isomorphisms required by Theorem \ref{main-0} i) and ii) is a
consequence of the following simple lemma, which concludes the proof of Theorem \ref{main-0}.

\begin{lem}\label{unicitate-coord}
Any automorphism of a germ $((M, p), \circ , e, E)$ of regular
$F$-manifolds is the identity map.
\end{lem}

\begin{proof} Let $\psi$ be such an automorphism. Then $\psi_{*}(E^{i}) = E^{i}$
for any $i\geq 0.$ From regularity,  $\psi_{*} (X) =
X$, i.e. $\phi_{t}^{X}\circ \psi = \psi \circ \phi_{t}^{X}$, where 
$\phi_{t}^{X}$ is the flow of $X$ and $X\in{\mathcal T}_{M}$ is arbitrary. Since $\psi (p) =p$, we obtain
that $\psi $ is the identity map.\end{proof}

In the following sections we develop  applications of Theorem
\ref{main-0}.

\section{Frobenius metrics in canonical
coordinates}\label{frobenius-sect}

In this section we study Frobenius metrics in the coordinate
system provided by Theorem \ref{main-0} i). In Subsection
\ref{sub-euler} we express the conditions  which involve the unit
and Euler fields. The flatness of the metric will be treated in
Subsection \ref{section-flatness}.

\subsection{The unit and Euler fields}\label{sub-euler}

Let $M:= \mathbb{C}^{m_{1}} \times \cdots \times
\mathbb{C}^{m_{n}}$. We denote by $(t^{i (\alpha )})$ ($0\leq
i\leq m_{\alpha } -1$, $1\leq \alpha \leq n$) the canonical
coordinates on $M$ and by $\{
\partial_{i(\alpha )}:= \frac{\partial}{\partial t^{i(\alpha
)}}\}$ the associated vector fields. According to Theorem
\ref{main-0} i), the multiplication
$$
\partial_{i(\alpha)} \circ \partial_{j(\beta )}
=\begin{cases} \partial_{(i+j)(\alpha )}, \quad  \alpha
=\beta ,\  i+j \leq m_{\alpha}-1\\
0,\quad \text{otherwise},\\
\end{cases}
$$
and the vector field
$$
E = \sum_{\alpha = 1}^{n}\left(  (t^{0(\alpha )}+ a_{\alpha })
\partial_{0(\alpha )} + (t^{1(\alpha )}+ 1)
\partial_{1(\alpha )} + \sum_{i=2}^{m_{\alpha} -1} t^{i(\alpha )}\partial_{i(\alpha
)}\right)
$$
give $M$ the structure of an $F$-manifold, with unit field
$$
e= \sum_{\alpha =1}^{n}\partial_{0(\alpha )}.
$$
Any multiplication
invariant metric on $M$ is of the form
\begin{equation}\label{metrica-g}
g= \delta_{\alpha\beta }\eta_{(i+j)(\alpha )}
dt^{i(\alpha )}\otimes dt^{j(\beta )}
\end{equation}
for some functions $\eta_{i(\alpha )}$, where $1\leq \alpha ,\beta
\leq n$, $0\leq i\leq m_{\alpha } -1$ and $\eta_{i(\alpha )}=0$,
for $i\geq m_{\alpha }$. (To simplify notation, in
(\ref{metrica-g}) and in other places we omit the summation sign).

\begin{prop}\label{local-prop} i) The coidentity $e^{\flat} := g(e, \cdot )$ is closed
if and only if there is  a function $H$ (called a metric
potential) such that $\eta_{i(\alpha )} =\partial_{i(\alpha )}(H)$
for any $i(\alpha ).$\

ii) The unit field $e$ is flat (with respect to the Levi-Civita
connection $\nabla^{\mathrm{LC}}$ of $g$) if and only if
$d(e^{\flat} )=0$ and $e(\eta_{i (\alpha )})=0$, for any $i(\alpha
).$\

iii) The Euler field rescales $g$ (i.e. $L_{E}(g) = Dg$ for a
constant $D$) if and only if $E(\eta_{i(\alpha )})= (D-2)
\eta_{i(\alpha )}$ for any $i(\alpha ).$

\end{prop}

\begin{proof}
Since $e= \sum_{\alpha =1}^{n}\partial_{0(\alpha )}$, the
coidentity is given by $e^{\flat} = \eta_{i(\alpha )} dt^{i(\alpha
)}$. It is closed if and only if it is exact, i.e. $e^{\flat} = d
H$, for a function $H.$ Claim i) follows. For claim ii), we use
that $\nabla^{\mathrm{LC}}(e)=0$ if and only if $d(e^{\flat})=0$
and $L_{e}(g)=0.$ But
\begin{align*}
L_{e} (g) (\partial_{i(\alpha )}, \partial_{j(\beta ) } ) &=
\delta_{\alpha \beta }e( \eta_{(i+j) (\alpha )} ) - g([e,
\partial_{i(\alpha )}], \partial_{j(\beta)})-  g(
\partial_{i(\alpha )}, [e,\partial_{j(\beta)}])\\
&= \delta_{\alpha \beta }e( \eta_{(i+j) (\alpha )} ),
\end{align*}
where in the second line we used $e= \sum_{\alpha
=1}^{n}\partial_{0(\alpha )}$ and the fact that the vector fields
$\{
\partial_{i(\alpha )}\}$ commute. Claim ii) follows. Claim iii)
follows equally easy.\end{proof}

\subsection{The flatness condition}\label{section-flatness}

It remains to study the flatness. For this, let us consider again
Dubrovin's description of  Frobenius metrics on constant (not
necessarily regular) $F$-manifolds, recalled in Subsection
\ref{constant-dub}. In Lemma \ref{general} we prove that the
rotation coefficient operator is determined (modulo a term
${\mathcal C}_{X}$, for $X\in {\mathcal T}_{N}$) by the Frobenius
metric. Therefore, the generalized Darboux-Egoroff equations
(\ref{gen-darboux}) may be written directly in terms of the metric
(rather than the rotation coefficient operator). Proposition
\ref{corolar-general} below is a rewriting of Theorem 3.1 of
\cite{dubrovin}. We will apply it in order to obtain a description
of Frobenius metrics on regular, globally nilpotent $F$-manifolds
(see Theorem \ref{flat-gen}).

We use the notation from Subsection \ref{constant-dub}. In
particular, we identify $TN$ with $T^{*}N$ using $\epsilon =
\epsilon_{ij}dt^{i}\otimes dt^{j}$. We denote by $\epsilon^{-1} :
T^{*}N \rightarrow TN$ this isomorphism. The induced metric on
$T^{*}N$ will also be denoted by $\epsilon .$ It is given by
$\epsilon = \epsilon^{ij} \partial_{i}\otimes
\partial_{j}$ where $(\epsilon^{ij})$ is the inverse of
$(\epsilon_{ij})$.

\begin{lem}\label{general}
Let $(N,\circ , e)$ be a constant $F$-manifold, with constant
multiplication invariant metric $\epsilon\in S^{2}(T^{*}N)$, $\psi
= \psi_{j} dt^{j}\in \Omega^{1} (N)$  an invertible $1$-form and
$T= \epsilon^{-1}(\psi )\in {\mathcal T}_{N}$ the $\epsilon$ dual
vector field. There is an $\epsilon$ symmetric endomorphism
$\tilde{\gamma} \in \mathrm{End}(TN)$ which satisfies
\begin{equation}\label{necesitate-1}
L_{\partial_{i}} (\psi ) = \psi [ {\mathcal C}_{i}, \tilde{\gamma}
],\ \forall i,
\end{equation}
if and only if $\epsilon (\psi , \psi )$ is constant and the
endomorphism $\gamma \in \mathrm{End}(TN)$, defined by
\begin{equation}\label{gamma-def}
\gamma = L_{\epsilon^{-1}(dt^{i})}(\psi )\otimes
(\partial_{i}\circ T^{-1}),
\end{equation}
is $\epsilon$ symmetric. If $\epsilon (\psi ,\psi )$ is constant
and $\gamma$ is $\epsilon$ symmetric, then $\gamma$ satisfies
(\ref{necesitate-1}), and, moreover,  any other $\epsilon$
symmetric endomorphism $\tilde{\gamma}$, which satisfies
(\ref{necesitate-1}), is of the form $\tilde{\gamma} = \gamma +
{\mathcal C}_{X}$, for $X\in {\mathcal T}_{N}$.
\end{lem}

\begin{proof}
We divide the proof into several steps.\

{\bf Step 1.} We claim that the operator $\gamma$ defined by
(\ref{gamma-def}) satisfies
\begin{equation}\label{gamma-T}
\gamma (T)= \frac{1}{2} \epsilon^{-1} (dt^{i}) (\epsilon (\psi ,
\psi )) \partial_{i} \circ T^{-1}.
\end{equation}
To prove (\ref{gamma-T}), we use the definition of $\gamma$ and
$L_{\partial_{i}} (\epsilon ) =0$:
\begin{align*}
\gamma (T) & = L_{\epsilon^{-1}(dt^{i})} (\psi ) (T)
\partial_{i}\circ T^{-1} = \epsilon ( L_{\epsilon^{-1}(dt^{i})} (\psi )
, \psi )\partial_{i}\circ T^{-1}\\
&= \frac{1}{2} \epsilon^{-1}(dt^{i}) (\epsilon (\psi , \psi ))
\partial_{i}\circ T^{-1}.
\end{align*}

{\bf Step 2.} We claim that any $\epsilon$ symmetric endomorphism
$\tilde{\gamma}$  of $TN$ satisfies
\begin{equation}\label{proprietate-generala}
\tilde{\gamma } = \epsilon^{ik} \psi [{\mathcal C}_{k},
\tilde{\gamma} ] (\partial_{j}) dt^{j}\otimes (\partial_{i} \circ
T^{-1}) + {\mathcal C}_{\tilde{\gamma}(T)\circ T^{-1}}.
\end{equation}
To prove (\ref{proprietate-generala}), let $\tilde{\gamma}$ be
such an endomorphism. Using that ${\mathcal C}_{i}$ and
$\tilde{\gamma}$ are $\epsilon$ symmetric, we obtain
\begin{equation}\label{calcul}
\psi [ {\mathcal C}_{i}, \tilde{\gamma} ] (\partial_{j}) =
\epsilon (T, [ {\mathcal C}_{i},\tilde{\gamma} ] (\partial_{j}))=
\epsilon (T\circ \tilde{\gamma} (\partial_{j}) - \tilde{\gamma}
(T) \circ\partial_{j},\partial_{i}).
\end{equation}
From $X = \epsilon^{ik}\epsilon (X, \partial_{k})
\partial_{i}$, for any $X\in {\mathcal T}_{N}$,  and relation (\ref{calcul}), we obtain
\begin{align*}
T\circ \tilde{\gamma} (\partial_{j}) - \tilde{\gamma} (T)\circ
\partial_{j}&= \epsilon^{ik} \epsilon ( T\circ \tilde{\gamma}
(\partial_{j}) - \tilde{\gamma} (T)\circ \partial_{j},\partial_{k}) \partial_{i}\\
&= \epsilon^{ik} \psi [{\mathcal C}_{k}, \tilde{\gamma} ]
(\partial_{j})\partial_{i},
\end{align*}
which implies (\ref{proprietate-generala}).\

{\bf Step 3.} We claim that if there is an $\epsilon$ symmetric
endomorphism $\tilde{\gamma}$ of $TN$, which satisfies
(\ref{necesitate-1}), then  $\epsilon (\psi , \psi )$ is constant
and the operator $\gamma$, defined by (\ref{gamma-def}),  is
$\epsilon$ symmetric. Let $\tilde{\gamma}$ be such an
endomorphism. From (\ref{necesitate-1}) and $L_{\partial_{i}}
(\epsilon )=0$,
\begin{equation}\label{partial-i}
\frac{1}{2} \partial_{i}( \epsilon (\psi , \psi ))= \epsilon
(L_{\partial_{i}} (\psi ), \psi ) = \epsilon ( \psi [ {\mathcal
C}_{i}, \tilde{\gamma }], \psi ) = -\epsilon ( [{\mathcal C}_{i},
\tilde{\gamma}](T), T) =0.
\end{equation}
(In the third equality we used that $[{\mathcal C}_{i},
\tilde{\gamma}]\in \mathrm{End}(TN)$ is $\epsilon$ skew-symmetric;
owing to this, the $1$-form $\psi [ {\mathcal C}_{i},
\tilde{\gamma }]\in \Omega^{1} (N)$ is $\epsilon$ dual to $- [
{\mathcal C}_{i}, \tilde{\gamma }](T)\in {\mathcal T}_{N}$. In the
fourth equality we used again that $[{\mathcal C}_{i},
\tilde{\gamma}]$ is $\epsilon$ skew-symmetric). Relation
(\ref{partial-i}) shows that $\epsilon (\psi , \psi )$ is
constant. Using (\ref{proprietate-generala}) ($\tilde{\gamma}$ is
$\epsilon$ symmetric), (\ref{necesitate-1}) and
$\epsilon^{-1}(dt^{i}) =\epsilon^{ik}\partial_{k}$, we obtain
\begin{equation}\label{t}
\tilde{\gamma} = \epsilon^{ik}\partial_{k} (\psi_{j} )
dt^{j}\otimes (\partial_{i}\circ T^{-1}) +{\mathcal
C}_{\tilde{\gamma} (T) \circ T^{-1}} = \gamma +{\mathcal
C}_{\tilde{\gamma} (T) \circ T^{-1}}.
\end{equation}
Since $\tilde{\gamma}$ is $\epsilon$ symmetric, so is $\gamma .$
Our claim follows.\

{\bf Step 4.}  We assume that  $\epsilon (\psi , \psi )$ is
constant and  $\gamma$ is $\epsilon$ symmetric. We claim that
$\gamma$ satisfies (\ref{necesitate-1}). Since $\epsilon (\psi ,
\psi )$ is constant, $\gamma (T)=0$ (see relation
(\ref{gamma-T})). Since $\gamma$ is symmetric and $\gamma (T)=0$,
relation (\ref{proprietate-generala}) implies that
\begin{equation}\label{doi}
\gamma = \epsilon^{ik} \psi [ {\mathcal C}_{k}, \gamma
](\partial_{j}) dt^{j}\otimes ( \partial_{i}\circ T^{-1}) .
\end{equation}
On the other hand, from its definition (\ref{gamma-def}),
\begin{equation}\label{doi-1}
\gamma = \epsilon^{ik} \partial_{k}(\psi_{j})dt^{j}\otimes
(\partial_{i}\circ T^{-1}).
\end{equation}
Combining (\ref{doi}) with (\ref{doi-1}) we obtain that
$\partial_{k} (\psi_{j} ) =\psi [ {\mathcal C}_{k}, \gamma
](\partial_{j})$, i.e.  $\gamma$ satisfies (\ref{necesitate-1}),
as claimed.\

{\bf Step 5.} In the hypothesis from Step 4, we claim that any
other $\epsilon$ symmetric endomorphism $\tilde{\gamma}$, which
satisfies (\ref{necesitate-1}), is equal to $\gamma + {\mathcal
C}_{\tilde{\gamma }(T) \circ T^{-1}}.$ Let $\tilde{\gamma}$ be
such an endomorphism. Since it is $\epsilon$ symmetric, it
satisfies (\ref{proprietate-generala}). Using
(\ref{necesitate-1}), relation (\ref{proprietate-generala})
becomes $\tilde{\gamma} = \gamma  +{\mathcal C}_{\tilde{\gamma}(T)
\circ T^{-1}}$, as needed.
\end{proof}

\begin{prop}\label{corolar-general}
Let $(N, \circ ,e )$ be a constant $F$-manifold and $\epsilon\in
S^{2} (T^{*}N)$ a constant, multiplication invariant metric. Let
$\psi =\psi_{j} dt^{j}\in \Omega^{1} (N)$ be an invertible
$1$-form. Then the metric
$$
g(X ,Y):= (\psi\circ \psi )(X\circ Y)
$$
is Frobenius on $(N, \circ , e)$ if and only if $\epsilon (\psi ,
\psi )$ is constant and the endomorphism
\begin{equation}
\gamma = L_{\epsilon^{-1}(dt^{i})}(\psi )\otimes
(\partial_{i}\circ T^{-1})
\end{equation}
is $\epsilon$ symmetric and satisfies the generalized
Darboux-Egoroff equations (\ref{gen-darboux}). Above $T =
\epsilon^{-1} (\psi )\in {\mathcal T}_{N}$ is  $\epsilon$ dual to
$\psi .$
\end{prop}

\begin{proof} From Theorem \ref{dub-thm},
$g$ is Frobenius if and only if there is an $\epsilon$ symmetric
endomorphism $\tilde{\gamma} \in \mathrm{End}(TN)$ (a rotation
coefficient operator), which satisfies (\ref{necesitate-1}) and
the generalized Darboux-Egoroff equations (\ref{gen-darboux}).
From Lemma \ref{general}, the existence of an $\epsilon$ symmetric
endomorphism  $\tilde{\gamma}$, which satisfies
(\ref{necesitate-1}), is equivalent to the $\epsilon$ symmetry of
$\gamma$ and to $\epsilon (\psi , \psi )$ being constant. Suppose
that these equivalent conditions hold. From Lemma \ref{general}
again, $\gamma = \tilde{\gamma} - {\mathcal
C}_{\tilde{\gamma}(T)\circ T^{-1}}$. Therefore, $\tilde{\gamma}$
satisfies the generalized Darboux-Egoroff equations if and only if
$\gamma$ does. Our claim follows.
\end{proof}

\begin{rem}\label{gamma-alternativ}{\rm There is  an alternative formula for the endomorphism $\gamma$
from Lemma \ref{general}, which is  more suitable for
computations. Let $c_{ij}^{k}$ and $c^{ij}_{k}$ be the structure
constants, in canonical coordinates $(t^{i})$, of the
multiplications on $TN$ and $T^{*}N$. We claim that
\begin{equation}\label{gamma-def-alt}
\gamma = \partial_{k}(\psi_{j}) \beta_{s} \epsilon^{ik} c_{i}^{st}
dt^{j}\otimes\partial_{t}  ,
\end{equation}
where $\beta \in \Omega^{1} (N)$ is the inverse of $\psi$.
Relation (\ref{gamma-def-alt}) is obtained from the following
computation: from  (\ref{gamma-def}),
\begin{align*}
\gamma  &= \epsilon^{ik}
\partial_{k} (\psi_{j}) dt^{j}\otimes(\partial_{i} \circ \epsilon^{-1}(\beta ))  =
\epsilon^{ik} \partial_{k} (\psi_{j}) \beta_{s}dt^{j}\otimes
(\partial_{i}\circ \epsilon^{-1} (dt^{s}))\\
&= \partial_{k} (\psi_{j}) \beta_{s}\epsilon^{ik} \epsilon^{sf}
dt^{j} \otimes (\partial_{i}\circ \partial_{f})  =
\partial_{k} (\psi_{j}) \beta_{s}\epsilon^{ik} \epsilon^{sf}
c_{if}^{t} dt^{j}\otimes\partial_{t}   \\
&= \partial_{k} (\psi_{j} )\beta_{s}\epsilon^{ik} c^{st}_{i}
dt^{j}\otimes\partial_{t},
\end{align*}
where we used $\epsilon^{-1} (dt^{s}) = \epsilon^{sf}
\partial_{f}$,  $T^{-1}= \epsilon^{-1}(\beta )$ and
$c_{i}^{st} = \epsilon^{sf} c_{if}^{t}$.}
\end{rem}

We now return to the setting of regular $F$-manifolds. For
simplicity, we assume that $(M, \circ , e ,E)$ is globally
nilpotent (and regular, of dimension $m$). Let $(t^{0},\cdots ,
t^{m-1})$ be the coordinate system of $M$ provided by Theorem
\ref{main-0} i) and $\epsilon \in S^{2}(T^{*}M)$  the
(multiplication invariant) metric given by
\begin{equation}\label{def-ep}
\epsilon = \epsilon_{ij} dt^{i}\otimes dt^{j},\
\epsilon_{ij}=\epsilon (\partial_{i},\partial_{j}):= \delta_{i+j,
m-1}.
\end{equation}
We identify $TM$ with $T^{*}M$ using $\epsilon .$ The induced
multiplication on $T^{*}M$ is given by $dt^{i} \circ dt^{j} =
dt^{i+j- (m-1)}$ (with the convention $dt^{s}=0$ when $s\geq m$ or
$s< 0$) and $dt^{m-1}$ is the unit. A $1$-form $\psi =
\psi_{j}dt^{j}\in \Omega^{1} (M)$ is invertible if and only if
$\psi_{m-1}$ is non-vanishing. If $\psi$ is invertible and $\beta
= \beta_{j} dt^{j}$ is its inverse, then
\begin{equation}
\beta_{m-1} \psi_{m-1} =1,\quad \sum_{r+s =k}\beta_{s} \psi_{r}
=0,\ m-1\leq  k< 2(m-1).
\end{equation}

The following theorem is our main result from this section.

\begin{thm}\label{flat-gen} Let $(M, \circ , e, E)$ be a regular,
globally nilpotent, $m$-dimensional $F$-manifold, with fixed
constant metric $\epsilon \in S^{2}(T^{*}M)$ given by
(\ref{def-ep}). Let $g$ be a multiplication invariant metric,
given by
\begin{equation}\label{g-i-j}
g = \eta_{i+j} dt^{i}\otimes dt^{j}.
\end{equation}
We fix a branch of $(\eta_{m-1})^{1/2}.$\

i) There is a unique invertible $1$-form $\psi =\psi_{j}dt^{j} \in
\Omega^{1} (M)$, related to $g$ by
\begin{equation}
g (X, Y)= (\psi\circ \psi )(X\circ Y),\ X, Y\in TM.
\end{equation}
Its $(m-1)$-component is given by $\psi_{m-1} =(\eta_{m-1})^{1/2}$
and its remaining components are determined inductively by the
conditions:
\begin{equation}
\sum_{s+t = (m-1) +k} \psi_{s}\psi_{t} = \eta_{k},\ 0\leq k\leq
m-2.
\end{equation}
ii) The metric $g$ is Frobenius on $(M, \circ , e)$ if and only if
$$
\epsilon (\psi, \psi ) = \sum_{i+j=m-1} \psi_{i}\psi_{j}
$$
is
constant and
$$
\gamma := \sum_{j}\sum_{i\leq s} \beta_{s}
\partial_{m-1-i} (\psi_{j})
\partial_{m-1+i-s}\otimes dt^{j}
$$
is $\epsilon$ symmetric and satisfies the generalized
Darboux-Egoroff equations (\ref{gen-darboux}), where $\beta :=
\beta_{j}dt^{j} \in \Omega^{1}(M)$ is the inverse of $\psi$. In
particular, if $g$ is Frobenius then there is (locally) a function
$H$ such that $\eta_{i} =\partial_{i}(H)$, for any $i$, and
$\partial_{i}(H)$ is independent of $t^{0}.$\

iii) The metric $g$ is Frobenius on $(M, \circ , e,E)$  if and
only if the conditions from ii) hold and, moreover, $E(\eta_{i})=
(D-2) \eta_{i}$ for any $i.$

\end{thm}

\begin{proof} The proof follows from Propositions \ref{local-prop} and
\ref{corolar-general} and relation (\ref{gamma-def-alt}).
\end{proof}

\begin{example}\label{2}{\rm We consider the setting of Theorem \ref{flat-gen}.

i) The metric $\epsilon$ itself is  Frobenius  with $\psi =
dt^{m-1}$ and $\gamma =0.$\

ii) Assume that $m=2$. The $1$-form $\psi$, its inverse $\beta$,
the operator $\gamma$ and $\epsilon (\psi , \psi )$ are given by
\begin{align*}
&\psi  = \frac{1}{2}\eta_{0}(\eta_{1})^{-1/2} dt^{0} +
(\eta_{1})^{1/2} dt^{1},\ \beta = -\frac{1}{2} (\eta_{1})^{-3/2}
\eta_{0}dt^{0} +
(\eta_{1})^{-1/2}dt^{1}\\
&\gamma (\partial_{i}) = \beta_{1}\partial_{1} (\psi_{i})
\partial_{0} + (\beta_{0}\partial_{1}(\psi_{i})+\beta_{1}
\partial_{0}(\psi_{i}))\partial_{1},\ 0\leq i\leq 1\\
& \epsilon (\psi , \psi )= 2 \psi_{0} \psi_{1} = \eta_{0}.
\end{align*}
The generalized Darboux-Egoroff equations reduce to $[ {\mathcal
C}_{1}, L_{\partial_{0}}(\gamma )]=0$. The unit field $e$ is flat
if and only if $\partial_{0}(\eta_{1}) =\partial_{1} (\eta_{0})$
and $\eta_{i}$ are independent of $t^{0}$ ($i=1,2$). Suppose that
$e$ is flat. Then $\eta_{0}$ is constant, the generalized
Darboux-Egoroff equations are automatically satisfied and $\gamma$
is $\epsilon$ symmetric. A metric is Frobenius on $(M, \circ , e)$
if and only if it is of the form $g = \dot{f} (dt^{0} \otimes
dt^{1}+ dt^{1} \otimes dt^{0})$, where $f= f(t^{1})$ and its
derivative $\dot{f}$ (with respect to $t^{1}$) is non-vanishing.
The metric $g$ is Frobenius on $(M, \circ ,e, E)$ if, moreover,
$t^{1} \ddot{f} = (D-2)\dot{f}$, for a constant $D\in
\mathbb{C}.$\

iii) Assume that $m=3$. The $1$-form $\psi= \psi_{j}dt^{j}$, its
inverse $\beta = \beta_{j}dt^{j}$ and $\epsilon (\psi , \psi )$
are given by
\begin{align*}
&\psi = (\frac{1}{2}\eta_{0}(\eta_{2})^{-1/2}-\frac{1}{8}
(\eta_{1})^{2}(\eta_{2})^{-3/2})dt^{0} +
\frac{1}{2}\eta_{1}(\eta_{2})^{-1/2} dt^{1} +(\eta_{2})^{1/2}
dt^{2}\\
&\beta =  ( -\frac{1}{2}\eta_{0} (\eta_{2})^{-3/2}+\frac{3}{8}
(\eta_{1})^{2}(\eta_{2})^{-5/2}) dt^{0} -\frac{1}{2} \eta_{1}
(\eta_{2})^{-3/2} dt^{1} +(\eta_{2})^{-1/2}dt^{2}\\
&\epsilon (\psi , \psi )= 2\psi_{0}\psi_{2} + (\psi_{1})^{2} =
\eta_{0}.
\end{align*}
Suppose that $e$ is flat. Like in the case $m=2$, $\eta_{i}$ are
independent of $t^{0}$ and $\partial_{i}(\eta_{j})
=\partial_{j}(\eta_{i})$, for any $i, j.$ In particular,
$\eta_{0}$ is constant. The operator $\gamma$ is given by: for any
$0\leq i\leq 2$,
$$
\gamma (\partial_{i}) = \partial_{2} (\psi_{i})
\beta_{2}\partial_{0} + (\partial_{2}(\psi_{i}) \beta_{1} +
\partial_{1}(\psi_{i})\beta_{2} )\partial_{1} +
(\beta_{0}\partial_{2}(\psi_{i})+ \beta_{1}
\partial_{1}(\psi_{i}))\partial_{2}.
$$
It is $\epsilon$ symmetric if and only if $\gamma_{10}
=\gamma_{21}$, $\gamma_{00}=\gamma_{22}$ and $\gamma_{01} =
\gamma_{12}$, where $\gamma (\partial_{i}) =
\gamma_{ji}\partial_{j}.$ Suppose that these relations are
satisfied. The generalized Darboux-Egoroff equations become the
highly non-trivial condition
$$
[ {\mathcal C}_{1} ,L_{\partial_{2}} (\gamma )] - [{\mathcal
C}_{2}, L_{\partial_{1}} (\gamma ) ] + [ [ {\mathcal C}_{1},
\gamma ], [{\mathcal C}_{2}, \gamma ]]=0,
$$
which, in terms of $\gamma_{ij}$, gives
\begin{align*}
&\partial_{2} (\gamma_{11} - \gamma_{00} ) -\partial_{1}
(\gamma_{01}) + (\gamma_{01})^{2} - (\gamma_{11} -
\gamma_{00})\gamma_{02}=0\\
&\partial_{2}(\gamma_{01}) -\partial_{1} (\gamma_{02}) -
\gamma_{02}\gamma_{01}=0\\
&\partial_{2} (\gamma_{02}) + (\gamma_{02})^{2}=0.
\end{align*}}
\end{example}

\section{Infinitesimal symmetries in canonical
coordinates}\label{lie-sim}

\begin{defn}\label{def-inf-dez} An infinitesimal symmetry of an $F$-manifold $(M,
\circ , e ,E)$ is a vector field $X$ which preserves the
multiplication and the Euler field:
$$
L_{X}(\circ ) =0,\quad [X, E]=0.
$$
\end{defn}

Using the Jacobi identity and the general formula $L_{[X, Y]} = [
L_{X}, L_{Y}]$ for the Lie derivative, we obtain that the set
$\mathcal L$ of infinitesimal symmetries of any $F$-manifold is a
subalgebra of the Lie algebra of vector fields. In this section we
compute the Lie algebra $\mathcal L$ of germs of  regular
$F$-manifolds. According to Theorem 2.11 of \cite{hert-book}, an
infinitesimal symmetry of a product $F$-manifold decomposes into a
product of infinitesimal symmetries of the factors. The Lie
algebra $\mathcal L$ decomposes accordingly and, from Theorem
\ref{main-0},  there is no loss of generality to assume that the
germ is the standard model $((\mathbb{C}^{m},0), \circ ,e, E)$,
with coordinates $(t^{0}, \cdots , t^{m-1})$ and $F$-manifold
structure given by (\ref{prima-form-1}) and (\ref{prima-form-2})
(with no index $\alpha$). We begin with the following lemma.

\begin{lem} A vector field $X$ on
$((\mathbb{C}^{m},0), \circ , e , E)$ satisfies $L_{X}(\circ)=0$
if and only if
\begin{align}
\nonumber& [\partial_0, X]=0,\quad  [\partial_1 , X]\circ
\partial_{m-1}=0,\\
\label{cond-circ} &[\partial_i, X]= i\partial_{i-1}\circ
[\partial_1 , X],\ 2\leq i\leq m-1.
\end{align}
\end{lem}

\begin{proof} For any vector field $X$,
$$L_{X}(\circ)(\partial_i,\partial_j) = \begin{cases}
[X,\partial_{i+j}]-[X,\partial_i]\circ\partial_j
-\partial_i\circ [X,\partial_j],\quad i+j\leq m-1,\\
-[X,\partial_i]\circ\partial_j -\partial_i\circ [X,\partial_j]
\quad i+j\geq m.
\end{cases}$$
In particular,
\begin{align*}
&L_X(\circ) (\partial_0,\partial_0) = [\partial_0,X],\\
&L_X(\circ) (\partial_1,\partial_{j-1}) =
[X,\partial_j]-[X,\partial_1]\circ \partial_{j-1}
-\partial_1\circ [X,\partial_{j-1}],\quad 2\leq j\leq m-1,\\
&L_X(\circ) (\partial_1,\partial_{m-1}) =
-[X,\partial_1]\circ\partial_{m-1}-\partial_1\circ
[X,\partial_{m-1}].
\end{align*}
By induction, we obtain that  the right hand side of these
relations vanish if and only if the relations (\ref{cond-circ})
hold. Moreover, if the relations (\ref{cond-circ}) hold, then
$L_{X}(\circ )(\partial_{i},
\partial_{j})=0$, for any $i, j$ (easy check).
\end{proof}

\begin{prop}\label{thm-inf-dez}
The system of vector fields $\{ Y_{1}, \cdots , Y_{m-1}\}$,
defined by
$$
Y_1 := (t^{1}+1) \partial_{1} +
\sum_{j=2}^{m-1}jt^j\partial_j,\quad Y_{k} := \partial_{k-1}\circ
Y_{1},\ 2\leq k\leq m-1,
$$
is a basis of the Lie algebra $\mathcal L$ of infinitesimal
symmetries of the standard model $((\mathbb{C}^{m},0), \circ , e,
E)$ and
\begin{equation}\label{check-direct}
[Y_i,Y_j] = \begin{cases} (i-j)Y_{i+j-1},\quad i+j\leq m,\\
0,\quad  i+j > m.\\
\end{cases}
\end{equation}
\end{prop}

\begin{proof}
It is easy to check that $Y_{1}$ satisfies the relations
(\ref{cond-circ}) and $[E, Y_{1}]=0$, i.e. $Y_{1}$ belongs to
$\mathcal L .$ Using that $Y_{1}\in \mathcal L$, $L_{\partial_{k}}
(\circ )=0$ and that $E$ is an Euler field, we obtain: for any
$k\geq 2$,
\begin{align*}
&[E, Y_{k}] = [ E, \partial_{k-1}\circ Y_{1}] = [E,
\partial_{k-1}] \circ Y_{1} + \partial_{k-1}\circ [E, Y_{1}] +
\partial_{k-1}\circ Y_{1}=0\\
&L_{Y_{k}} (\circ ) = L_{\partial_{k-1}\circ Y_{1}} (\circ ) =
\partial_{k-1} \circ L_{Y_{1}} (\circ ) + Y_{1} \circ
L_{\partial_{k-1}}(\circ )=0.
\end{align*}
We proved that $Y_{k}\in \mathcal L$, for any $k\geq 1.$ Relation
(\ref{check-direct}) can be checked directly.

Consider now an arbitrary vector field $X\in {\mathcal L}$. We
write it as $X= f_{0}\partial_{0} + f_{1}Y_{1} +\cdots + f_{m-1}
Y_{m-1}$, where $f_{k}$ are functions. We will prove that
$f_{0}=0$ and $f_{k}$ are constant, for any $k\geq 1.$ For any $s$
(sufficiently close to 0), $\Phi_s^X$ is an automorphism of the
$F$-manifold. Within the $F$-manifold, the hypersurfaces $\{t\,
|\, t^0+a=const\}$ are the subspaces where the only eigenvalue of
${\mathcal U}= E\circ$, namely $t^0+a$, is constant. As $\Phi_s^X$
is an automorphism which respects multiplication and Euler field,
it does not change this eigenvalue. Therefore the flow  of $X$
respects the hypersurfaces $\{t\, |\, t^0+a=const\}$ and we obtain
that $f_0=0$. Now, subtracting from $X$ a suitable linear
combination of $Y_1,...,Y_{m-1}$ (with constant coefficients), we
can suppose that $f_1(0)=...=f_{m-1}(0)=0$. But then the flow
$\Phi_s^X$, for any $s$, fixes the point $0$, so it is an
automorphism of the germ $((\mathbb{C}^m,0), \circ , e, E)$. By
Lemma \ref{unicitate-coord}, $\Phi_s^X=\mathrm{Id}$. This implies
$X=0$.
\end{proof}

\begin{rem}{\rm  The $F$-manifold in Theorem 2 i),
$(\mathbb{C}^m,\circ,e,E)$ with $\circ,e$ and $E$ given there
(with no index $\alpha$), is regular and globally nilpotent on
$\mathbb{C}\times (\mathbb{C}-\{1\})\times \mathbb{C}^{m-2}$. By
Theorem 2 ii), for any two values $t_1,t_2\in \mathbb{C}\times
(\mathbb{C}-\{1\})\times \mathbb{C}^{m-2}$ with $t^0_1=t^0_2$, the
germs $((\mathbb{C}^m, t_{1}),\circ,e,E)$ and $((\mathbb{C}^m,
t_{2}),\circ,e,E)$ are isomorphic, and the isomorphism is unique.
The flows $\Phi_s^X$ of the infinitesimal vector fields $X$ in
Proposition 30 realize these isomorphisms for nearby germs.}
\end{rem}

\section{Regular $F$-manifolds and meromorphic
connections}\label{merom-sect}

Let $\nabla^{0}$ be a meromorphic connection on the trivial vector
bundle $V^{0} = D \times \mathbb{C}^{n} \rightarrow  D$ (where $D$
is a small disc around the origin in $\mathbb{C}$), with
connection form $\Omega^{0}$ given by (\ref{omega-0}), in the
standard trivialization of $V^{0}$.  We assume that $B_{0}^{o}\in
M_{n}(\mathbb{C})$ is regular. Let $M^{\mathrm{can}}=
M^{\mathrm{can}} (B_{0}^{o}, B_{\infty})$ be the parameter space
of the Malgrange universal deformation $\nabla^{\mathrm{can}}$ of
$\nabla^{0}$ (see Definition \ref{def-para}). Recall that it is
the maximal integrable submanifold of the distribution $\mathcal
D\vert_{W}$, defined by (\ref{rond}), passing through $0$. The
tangent bundle $TM^{\mathrm{can}}$ admits a natural multiplication
$\circ_{\mathrm{can}}$: for any $\Gamma \in M^{\mathrm{can}}$,
$(\circ_{\mathrm{can}})_{\Gamma}$, acting on
$T_{\Gamma}M^{\mathrm{can}} = {\mathcal D}_{\Gamma}\subset
M_{n}(\mathbb{C})$, is the multiplication of matrices (it
preserves ${\mathcal D}_{\Gamma}$). It is clear that
$\circ_{\mathrm{can}}$ is associative, commutative, with unit
field $(\mathrm{Id}_{\mathrm{can}})_{\Gamma } = \mathrm{Id}$ (the
identity matrix), for any $\Gamma \in M^{\mathrm{can}}.$

\begin{prop}\label{model-uni} i) The multiplication $\circ_{\mathrm{can}}$ gives $M^{\mathrm{can}}$ the structure of a (regular) $F$-manifold,
with Euler field
\begin{equation}\label{euler-reg}
(E_{\mathrm{can}})_{\Gamma }:= - (B_{0})_{\Gamma}= - B_{0}^{o} +
\Gamma - [B_{\infty},\Gamma ],\quad  \Gamma \in M^{\mathrm{can}}.
\end{equation}

ii) Conversely, let $(M, \circ ,e,E)$ be a regular $F$-manifold,
$p\in M$, and $-B^{o}_{0}$ the representation of ${\mathcal U}_{p}
: T_{p}M \rightarrow T_{p}M$, ${\mathcal U}_{p}(X) = X\circ
E_{p}$, in a basis of $T_{p}M.$ Let $B_{\infty}$ be any matrix and
$M^{\mathrm{can}} := M^{\mathrm{can}} (B_{0}^{o}, B_{\infty})$.
The germs  $((M, p),\circ , e, E)$ and
$((M^{\mathrm{can}},0),\circ_{\mathrm{can}} ,
\mathrm{Id}_{\mathrm{can}},E_{\mathrm{can}})$   are isomorphic.
\end{prop}

\begin{proof}
Let $V= M^{\mathrm{can}} \times \mathbb{C}^{n} \rightarrow
M^{\mathrm{can}}$ be the trivial bundle. Elements of $V$ are pairs
$(\Gamma , v)$ where $\Gamma \in M^{\mathrm{can}}$ and $v\in
\mathbb{C}^{n}$. We shall denote by $V_{\Gamma} = \mathbb{C}^{n}$
the fiber of $V$ at $\Gamma\in  M^{\mathrm{can}}.$ From
Proposition \ref{F-mer} and relation (\ref{omega-cann}),
$\nabla^{\mathrm{can}}$ induces a Saito structure $(D , \Phi,
B_{0}, - B_{\infty})$ on $V$, as follows: $D$ is the canonical
flat connection (the constant sections of $V$ are $D$-flat); $\Phi
\in \Omega^{1}(M^{\mathrm{can}},\mathrm{End}( V))$ is given by
$\Phi_{X} =X\in M_{n}(\mathbb{C}) = \mathrm{End}(V_{\Gamma})$, for
any $X\in T_{\Gamma}(M^{\mathrm{can}}) \subset
M_{n}(\mathbb{C}^{n})$ (i.e. for any $v\in
V_{\Gamma}=\mathbb{C}^{n}$, $\Phi_{X}(v) = X(v)$ is the action of
the matrix $X$ on the vector $v$); $(B_{0})_{\Gamma},
(B_{\infty})_{\Gamma}\in \mathrm{End}(V_{\Gamma})$ are given by
$$
(B_{0})_{\Gamma }= B^{o}_{0} - \Gamma + [B_{\infty}, \Gamma ],\
(B_{\infty})_{\Gamma}= B_{\infty}.
$$
Let $v\in \mathbb{C}^{n}$ be a cyclic vector for $B_{0}^{o}$ and
$s\in \Gamma (V)$ the associated constant section. Thus, $s:
M^{\mathrm{can}} \rightarrow V = M^{\mathrm{can}}\times
\mathbb{C}^{n}$, $s(\Gamma ) = (\Gamma , v)$, for any $\Gamma \in
M^{\mathrm{can}}.$ The map
$$
I : TM^{\mathrm{can}}\rightarrow V,\  I(X) :=\Phi_{X}(s)= (\Gamma
, X(v)),\ X\in T_{\Gamma} M^{\mathrm{can}}
$$
is an isomorphism. From the definition of $\circ_{\mathrm{can}}$
and $E_{\mathrm{can}}$, $\Phi_{X\circ_{\mathrm{can}}Y} (s)=
\Phi_{X} \Phi_{Y}(s)$ and $\Phi_{E_{\mathrm{can}}}(s)= -
B_{0}(s)$, i.e. $E_{\mathrm{can}} = - I^{-1} B_{0}(s)$. It follows
that $(\circ_{\mathrm{can}},E_{\mathrm{can}})$ is induced from the
Saito bundle $(V,D , \Phi, B_{0}, - B_{\infty})$, as in
Proposition \ref{F-saito}. In particular, $(M^{\mathrm{can}},
\circ_{\mathrm{can}}, \mathrm{Id}_{\mathrm{can}},
E_{\mathrm{can}})$ is a (regular) $F$-manifold, as required. This
proves  claim {\it i)}.

For claim {\it ii}), let ${\mathcal U}_{\mathrm{can}} \in
\mathrm{End}(TM^{\mathrm{can}})$ be defined by ${\mathcal
U}_{\mathrm{can}} (X):= X\circ_{\mathrm{can}} E_{\mathrm{can}}.$
From Proposition \ref{F-saito}, $({\mathcal
U}_{\mathrm{can}})_{0}$ is conjugated to $- (B_{0})_{0}= -
B_{0}^{o}$. Since $ - B_{0}^{o}$ is the representation of
${\mathcal U}_{p}$ in a basis of $T_{p}M$, $({\mathcal
U}_{\mathrm{can}})_{0}$ and ${\mathcal U}_{p}$ belong to the same
conjugacy class. We conclude with Theorem \ref{main-0}.
\end{proof}

\begin{cor}\label{mer-rel} Any regular $F$-manifold $(M, \circ , e, E)$
is the parameter space of an integrable deformation of a
meromorphic connection on $V^{0} = D\times
\mathbb{C}^{n}\rightarrow D$, in Birkhoff normal form, with a pole
of Poincar\'{e} rank one in the origin.\end{cor}

\begin{proof}
Trivial, from Proposition \ref{model-uni} ii).
\end{proof}

\section{Initial conditions for Frobenius
metrics}\label{initial}

In this section we prove an initial condition theorem for
Frobenius metrics on regular $F$-manifolds (see Theorem
\ref{question} below). Our argument relies on Theorem \ref{main-0}
and the theory developed in \cite{hert-man-unfol}. A self
contained proof for the existence of a Frobenius metric with given
initial condition, which avoids the technicalities of
\cite{hert-man-unfol}, will be presented in Section
\ref{appendix}. The following remark justifies the properties of
${\mathcal V}_{p}$ from Theorem \ref{question}.

\begin{rem}{\rm If $(M, \circ , e, E, g)$ is a Frobenius manifold
and $L_{E}(g) = Dg$ then $\nabla^{\mathrm{LC}} E = {\mathcal V}+
\frac{D}{2}\mathrm{Id}$, where $\nabla^{\mathrm{LC}}$ is the
Levi-Civita connection of $g$ and $\mathcal V$ is the $g$
skew-symmetric part of $\nabla^{\mathrm{LC}}E.$ Using $[e, E] = e$
and $\nabla^{\mathrm{LC}} (e)=0$, we obtain ${\mathcal V}(e) =
(1-\frac{D}{2}) e.$}
\end{rem}

Our main result from this section is the following.

\begin{thm}\label{question} Let $(M, \circ , e, E)$ be a regular
$F$-manifold and $p\in M.$ Suppose that $g_{p}\in
S^{2}(T_{p}^{*}M)$ and ${\mathcal V}_{p}\in \mathrm{End}(T_{p}M)$
are given, such that the following conditions are satisfied:\

i) $g_{p}$ is multiplication invariant and non-degenerate;\

ii) ${\mathcal V}_{p}$ is $g_{p}$ skew-symmetric and ${\mathcal
V}_{p}(e_{p}) = (1-\frac{D}{2}) e_{p}$, for $D\in \mathbb{C}$.\

Then $g_{p}$ can be extended to a unique  Frobenius metric $g$ on
the germ $((M, p), \circ , e, E)$, such that
$(\nabla^{\mathrm{LC}} E)\vert_{T_{p}M} =\mathcal
V_{p}+\frac{D}{2}\mathrm{Id}$.
\end{thm}

\begin{proof}
We consider the linear data $(T_{p}M, {\mathcal U}_{p}, {\mathcal
V}_{p}, g_{p})$ (as usual, $\mathcal U_{p}$ is the multiplication
by $E_{p}$). From regularity, $e_{p}$ together with ${\mathcal
U}_{p}^{k}(e_{p})$ ($k\geq 1$), generate $T_{p}M$. Therefore, we
can apply Theorem 4.5 of \cite{hert-man-unfol}, with the Frobenius
type structure reduced to the vector space $(T_{p}M, {\mathcal
U}_{p}, {\mathcal V}_{p}, g_{p})$ and $\tau := e_{p}$ (see also
Remark 4.6 of \cite{hert-man-unfol}). We obtain a germ of
Frobenius manifolds $((\tilde{M},\tilde{p}), \tilde{\circ},
\tilde{e}, \tilde{E}, \tilde{g})$, with $L_{\tilde{E}}
(\tilde{g})= D\tilde{g}$, and an isomorphism
\begin{equation}\label{iso-j}
j: (T_{p}M,e_{p}, {\mathcal U}_{p}, {\mathcal V}_{p}, g_{p})
\rightarrow (T_{\tilde{p}}\tilde{M}, \tilde{e}_{\tilde{p}},
\tilde{\mathcal U}_{\tilde{p}},
(\tilde{\nabla}^{\mathrm{LC}}\tilde{E})\vert_{T_{\tilde{p}}
\tilde{M}}-\frac{D}{2}\mathrm{Id},\tilde{g}_{\tilde{p}})
\end{equation}
(where $\tilde{\mathcal U}_{\tilde{p}}$ is the multiplication by
$\tilde{E}_{\tilde{p}}$ and $\tilde{\nabla}^{\mathrm{LC}}$ is the
Levi-Civita connection of $\tilde{g}$). Since $j(e_{p}) =
\tilde{e}_{\tilde{p}}$ and $j \circ {\mathcal U}_{p} =
\tilde{\mathcal U}_{\tilde{p}} \circ j$, we obtain that
$j(E_{p}^{k}) = \tilde{E}_{\tilde{p}}^{k}$, for any $k\geq 0.$
Since ${\mathcal U}_{p}$ and $\tilde{\mathcal U}_{\tilde{p}}$ are
conjugated, the germs $((M,p), \circ ,e ,E)$ and
$((\tilde{M},\tilde{p}),\tilde{\circ},\tilde{e}, \tilde{E})$ are
isomorphic (from Theorem \ref{main-0}). Let $f: ((M,p), \circ ,e
,E) \rightarrow ((\tilde{M},\tilde{p}),\tilde{\circ},\tilde{e},
\tilde{E})$ be an isomorphism and $g:= f^{*}\tilde{g}$. The metric
$g$ is Frobenius on $((M, p),\circ , e, E)$. Since $f_{*}(e) =
\tilde{e}$, $f_{*} (E) = \tilde{E}$ and $f_{*}$ preserves
multiplications, $f_{*}(E^{k})= \tilde{E}^{k}$, for any $k\geq 0.$
In particular, $(f_{*})_{p}(E^{k}_{p})= \tilde{E}^{k}_{\tilde{p}}$
and hence $(f_{*})_{p}= j.$ It follows that $g\vert_{T_{p}M\times
T_{p}M} = j^{*} (\tilde{g}_{\tilde{p}}) = g_{p}$, i.e. $g$ extends
$g_{p}$. The Levi-Civita connections $\nabla^{\mathrm{LC}}$ and
$\tilde{\nabla}^{\mathrm{LC}}$ are related by
$$
f_{*} \nabla^{\mathrm{LC}}_{X} (Y) =
\tilde{\nabla}^{\mathrm{LC}}_{f_{*}(X)} f_{*}(Y),\quad  X, Y\in
{\mathcal T}_{M}.
$$
Applying this relation to $Y: = E$, using that $f_{*}(E)
=\tilde{E}$, $(f_{*})_{p}= j$ and
$$
((\tilde{\nabla}^{\mathrm{LC}}\tilde{E})
\vert_{T_{\tilde{p}}\tilde{M}}- \frac{D}{2} \mathrm{Id} )\circ j =
j\circ {\mathcal V}_{p}
$$
(from (\ref{iso-j})), we obtain
$$
\nabla^{\mathrm{LC}}_{X_{p}}(E) = j^{-1}
\tilde{\nabla}^{\mathrm{LC}}_{j(X_{p})} (\tilde{E} ) = {\mathcal
V}_{p}(X_{p} ) +\frac{D}{2} X_{p},\ X_{p}\in T_{p}M,
$$
as required. The existence of the extension is proved.

The unicity follows also from Theorem 4.5 of
\cite{hert-man-unfol}. More precisely, from this theorem we know
that any two extensions of $\tilde{g}_{p}$, with the required
properties, are related by an isomorphism of the germ $((M,p),
\circ , e, E)$. But any such  isomorphism is the identity map (see
Lemma \ref{unicitate-coord}). Our claim follows.
\end{proof}

\begin{rem}\label{non-weak}{\rm
In \cite{hert-man} (Chapter 3) it was asked whether there exist
F-manifolds which do not admit, in the neighborhood of any point,
any  Frobenius metric. There are F-manifolds for which the answer
to this question is not known (e.g. some generically semisimple
F-manifolds near points where they are not semisimple). Below we
describe two sources of examples for which the answer is negative.

a) Proposition 5.32 and Remark 5.33 of \cite{hert-book} provide
examples of germs $(M,0)$ of generically semisimple F-manifolds
such that $T_0M$ is a local algebra, but not a Frobenius algebra,
so it does not allow a nondegenerate multiplication invariant
metric. In Proposition 5.32 of \cite{hert-book} the F-manifolds
are 3 dimensional, and $T_0M$ is as an algebra isomorphic to
$\mathbb{C}\{x,y\}/(x^2,xy,y^2)$.

b) There are examples of (globally nilpotent) $F$-manifolds which
do not support any Frobenius metric. Such $F$-manifolds are
described in \cite{teleman}, Sections 2.5.2 and 2.5.3. Recall that
an associative, commutative, with unit multiplication $\circ$ on
the tangent bundle $TM$ of a manifold $M$ defines a (possible
non-reduced) subvariety $Y$ of $T^{*}M$, the spectral cover,  by
the ideal $I = (y^{0}-1, y^{i}y^{j} - \sum_{k}
a_{ij}^{k}(x)y^{k})\subset {\mathcal O}_{T^{*}M}$, where $(x^{i})$
are coordinates on $M$, with $\partial_{0} =e$ the unit field,
$(x^{i}, y^{j})$  are the induced coordinates on $T^{*}M$ and
$a_{ij}^{k}$ are defined by $\partial_{i} \circ \partial_{j} =
a_{ij}^{k} \partial_{k}$. The integrability condition
(\ref{integr}) from the definition of $F$-manifolds is equivalent
to $\{I,I\} \subset I$, where $\{ \cdot , \cdot \}$ is the
canonical Poisson bracket of $T^{*}M$ (see Theorem 2.5 of
\cite{teleman}). The reduced variety $Y_{\mathrm{red}}$, defined
by $\sqrt{I}$, is the support of the Higgs bundle $(TM, {\mathcal
C}_{X}(Y) = X\circ Y)$:
$$
Y_{\mathrm{red}}=\cup_{x\in M} \{ \lambda \in T^{*}_{x}M,\quad
\forall X\in T_{x}M,\ \mathrm{ker} ( {\mathcal C} _{X}-\lambda (X)
\mathrm{id} : T_{x}M \rightarrow T_{x}M)\neq 0\} .
$$
If the F-manifold can be enriched to a Frobenius manifold (even
without Euler field), this induces on the pull back of $T^*M$ to
$\mathbb{C}\times M$  a (T)-structure (in the notation of
\cite{hert}) respectively a holonomic $\mathcal{R}_\mathcal{X}$
module (in the notation of \cite{sabbah-2}, where $X=M$). This is
essentially the construction of the Saito bundle from the
Frobenius manifold, but without the data from the Euler field. A
result of Sabbah (\cite{sabbah-2}, Proposition 1.2.5) on holonomic
$\mathcal{R}_\mathcal{X}$-modules says that the reduced variety
$Y_{\mathrm{red}}$ is Lagrangian, or, equivalently, $\{\sqrt{I} ,
\sqrt{I} \}\subset \sqrt{I}$. The ideals defining the spectral
covers in the examples of $F$-manifolds from \cite{teleman},
mentioned above, do not satisfy this last condition. Thus, these
$F$-manifolds do not support any Frobenius metric.}
\end{rem}

\section{Appendix: proof of Theorem \ref{question} revised}\label{appendix}

As promised in Section \ref{initial}, we develop here an
alternative argument for the existence part in Theorem
\ref{question}. Consider the setting from this theorem. Let
$B_{0}^{o}$ and $B_{\infty}\in M_{n}(\mathbb{C})$ be the matrix
representations of ${\mathcal U}_{p} = ({\mathcal C}_{E})_{p}$ and
$\mathcal V_{p}$ in the basis $\mathcal B := \{ e_{p}, E_{p},
\cdots , E_{p}^{n-1}\}$ of $T_{p}M$ (where $n:= \mathrm{dim}(M)$):
$$
{\mathcal U}_{p}(E_{p}^{i}) = (B^{o}_{0})_{ji}E_{p}^{j},\
{\mathcal V}_{p} (E_{p}^{i}) = (B_{\infty})_{ji} E_{p}^{j}.
$$
For any $0\leq i\leq n-2$,
\begin{equation}\label{componente}
(B^{o}_{0})_{i+1,i}=1,\quad (B^{o}_{0})_{ji}=0,\   j\neq i+1.
\end{equation}
Since ${\mathcal V}_{p}(e_{p}) = ( 1-\frac{D}{2}) e_{p}$,
\begin{equation}\label{b-infty}
(B_{\infty})_{j0} =  (1-\frac{D}{2}) \delta_{0j},\ 0\leq j\leq
n-1.
\end{equation}
From the skew-symmetry of $\mathcal V_{p}$,
\begin{equation}\label{sa-adaug}
(B_{\infty})_{ki}e^{\flat}_{p} (E_{p}^{k+j}) + (B_{\infty})_{kj}
e^{\flat}_{p} (E_{p}^{k+i})=0,\quad 0\leq i, j\leq n-1,
\end{equation}
where $e^{\flat}_{p} (X) = g_{p}(e_{p}, X)$, for any $X\in
T_{p}M$.

Let $M^{\mathrm{can}} := M^{\mathrm{can}} (-B^{o}_{0}, -
B_{\infty})$, with its $F$-manifold structure provided by
Proposition \ref{model-uni}. From this proposition, we know that
there is an isomorphism
\begin{equation}\label{definition-f}
f:( (M,p) ,\circ , e, E) \rightarrow ((M^{\mathrm{can}},0),
\circ_{\mathrm{can}},
\mathrm{Id}_{\mathrm{can}},E_{\mathrm{can}}).
\end{equation}
Recall that the Euler field of $M^{\mathrm{can}}$ is given by
$$
(E_{\mathrm{can}})_{\Gamma}:= B^{o}_{0} + \Gamma + [ B_{\infty},
\Gamma ],\ \Gamma \in M^{\mathrm{can}}.
$$
In particular, $(E_{\mathrm{can}})_{0} = B^{o}_{0}$ and, since
$f_{*}(E^{i}) = (E_{\mathrm{can}})^{i}$, we obtain that
\begin{equation}\label{ultima-relatie}
(f_{*})_{p} (E^{i}_{p}) =(B^{o}_{0})^{i},\quad i\geq 0.
\end{equation}
Let
$$
(g_{\mathrm{can}} )_{0}: T_{0}M^{\mathrm{can}} \times T_{0}
M^{\mathrm{can}} \rightarrow \mathbb{C},\quad
(g_{\mathrm{can}})_{0}:= (f_{*}^{-1}) (g_{p})
$$
be the push-forward metric, given by
\begin{equation}\label{g-can-def}
(g_{\mathrm{can}})_{0} ( (B^{o}_{0})^{i}, (B^{o}_{0})^{j}) = g_{p}
(E_{p}^{i}, E_{p}^{j}) = e^{\flat}_{p} (E_{p}^{i+j}),\quad 0\leq
i, j\leq n-1.
\end{equation}
The endomorphism $({\mathcal U}_{\mathrm{can}})_{0}(X) =
X\circ_{\mathrm{can}} E_{\mathrm{can}}$ of $T_{0}
M^{\mathrm{can}}$ is the multiplication by $B^{o}_{0}\in
M_{n}(\mathbb{C})$ on $T_{0}M^{\mathrm{can}} \subset
M_{n}(\mathbb{C}).$ It is $(g_{\mathrm{can}})_{0}$ symmetric.
Using the isomorphism (\ref{definition-f}), the existence part in
Theorem \ref{question} is a consequence of the following lemma.

\begin{lem}The metric $(g_{\mathrm{can}})_{0}$ defined by (\ref{g-can-def}) admits an extension to a
Frobenius metric $g_{\mathrm{can}}$ on the germ
$((M^{\mathrm{can}}, 0), \circ_{\mathrm{can}},
\mathrm{Id}_{\mathrm{can}}, E_{\mathrm{can}})$, such that
\begin{equation}\label{final-dem}
(D^{\mathrm{LC}}E_{\mathrm{can}})_{0}   = (f_{*})_{p}\circ
{\mathcal V}_{p} \circ (f^{-1}_{*})_{p} + \frac{D}{2}\mathrm{Id}.
\end{equation}
Above $D^{\mathrm{LC}}$ is the Levi-Civita connection of
$g_{\mathrm{can}}$.
\end{lem}

\begin{proof} We preserve the notation from the proof of
Proposition \ref{model-uni}. Let $V= M^{\mathrm{can}} \times
\mathbb{C}^{n} \rightarrow M^{\mathrm{can}}$ be the trivial bundle
over $M^{\mathrm{can}}= M^{\mathrm{can}} (- B_{0}^{o}, -
B_{\infty})$ and $s \in \Gamma (V)$ the constant section $s(\Gamma
) = (\Gamma , v_{0})$, where $v_{0}:= (1,0,\cdots , 0)\in
\mathbb{C}^{n}$. We denote by $v_{1}:= (0,1,0,\cdots , 0)$,
$v_{2}:= (0,0,1,0,\cdots , 0)$, ..., $v_{n-1} = (0,\cdots , 0,1)$
the remaining standard vectors of $\mathbb{C}^{n}$. The
$F$-manifold structure of $M^{\mathrm{can}}$ is obtained (as
explained in Proposition \ref{F-saito}) from the Saito bundle
$(V,D , \Phi, B_{0}, B_{\infty})$ (defined as in the proof of
Proposition \ref{model-uni}, with $B_{0}^{o}$ replaced by
$-B_{0}^{o}$ and $B_{\infty}$ by $- B_{\infty}$), by means of the
isomorphism
\begin{equation}\label{def-I}
I: T M^{\mathrm{can}} \rightarrow V,\ I (X)=\Phi_{X}(s)=(\Gamma ,
X(v_{0})),\  X\in T_{\Gamma }M^{\mathrm{can}}\subset M_{n}
(\mathbb{C}).
\end{equation}
From (\ref{componente}), $v_{0}$ is a cyclic vector for
$B_{0}^{o}$ and
\begin{equation}\label{a-o-i}
I_{0} : T_{0}M^{\mathrm{can}}\rightarrow V_{0}=\mathbb{C}^{n},\
I_{0}((B^{o}_{0})^{i} )=(B_{0}^{o})^{i}(v_{0}) = v_{i},\quad 0\leq
i\leq n-1.
\end{equation}
Let $g_{0}:= (I^{-1}_{0})^{*}(g_{\mathrm{can}})_{0}\in S^{2}
(V_{0}^{*})$ be the push-forward of $(g_{\mathrm{can}})_{0}\in
S^{2} (T^{*}_{0}M^{\mathrm{can}} )$:
\begin{equation}\label{g-0}
g_{0} (v_{i},v_{j}) := (g_{\mathrm{can}})_{0} (I^{-1}_{0}(v_{i}),
I^{-1}_{0}(v_{j}))= (g_{\mathrm{can}})_{0} ( (B^{o}_{0})^{i},
(B^{o}_{0})^{j}) =e^{\flat}_{p} (E_{p}^{i+j}).
\end{equation}
Since $({\mathcal U}_{\mathrm{can}})_{0}$ is
$(g_{\mathrm{can}})_{0}$ symmetric (as stated before the lemma),
$I_{0} \circ ({\mathcal U}_{\mathrm{can}})_{0} \circ I^{-1}_{0}$
is $g_{0}$ symmetric. But $I_{0} \circ ({\mathcal
U}_{\mathrm{can}})_{0} \circ I^{-1}_{0} = - (B_{0})_{0}$ (see
Proposition \ref{F-saito}). Since $B_{0}\in \mathrm{End}(V)$ is
given by
\begin{equation}\label{given-b-0}
(B_{0})_{\Gamma } = - ( B^{o}_{0}+ \Gamma + [B_{\infty},
\Gamma]),\quad \Gamma \in M^{\mathrm{can}},
\end{equation}
we obtain that $- (B_{0})_{0}= B^{o}_{0}$. Therefore,
$B^{o}_{0}\in M_{n}(\mathbb{C})$ is $g_{0}$ symmetric. From
(\ref{sa-adaug}) and (\ref{g-0}), $B_{\infty}\in
M_{n}(\mathbb{C})$ is $g_{0}$ skew-symmetric. Since $B^{o}_{0}$ is
$g_{0}$ symmetric and $B_{\infty}$ is $g_{0}$ skew-symmetric,
$(B_{0})_{\Gamma}$ is $g_{0}$ symmetric when $\Gamma$ is so.

Let $M^{\mathrm{sym}}_{n}(\mathbb{C})$ be the manifold of $g_{0}$
symmetric matrices. We claim  that $(M^{\mathrm{can}},0) \subset
(M^{\mathrm{sym}}_{n}(\mathbb{C}),0)$. For this, we use the above
observation (namely, $(B_{0})_{\Gamma}\in
M^{\mathrm{sym}}_{n}(\mathbb{C})$ when $\Gamma \in
M^{\mathrm{sym}}_{n}(\mathbb{C})$) and the following general fact
(which can be easily checked): if $\mathcal D$ is an integrable
distribution on a manifold $M$,  $N$ is a submanifold of $M$ such
that ${\mathcal D}\vert_{N}\subset TN$ and $I^{\mathrm{max}}$ is
the maximal integrable submanifold of $\mathcal D$, which contains
$p\in N$, then there is a neighborhood $U$ of $p$ in $M$, such
that $I^{\mathrm{max}}\cap U\subset N\cap U$. Applying this fact
to $M:= W$ (a small open neighborhood of $0\in
M_{n}(\mathbb{C})$), $N:= M^{\mathrm{sym}}_{n}(\mathbb{C})\cap W$
and the distribution $\mathcal D\vert_{W}$ whose maximal
integrable submanifold is $M^{\mathrm{can}}$ (and whose fiber at
$\Gamma \in W$ is the vector space of polynomials in
$(B_{0})_{\Gamma}$, with $(B_{0})_{\Gamma}$ as in
(\ref{given-b-0})), we obtain $(M^{\mathrm{can}},0) \subset
(M^{\mathrm{sym}}_{n}(\mathbb{C}),0)$, as needed.

Let $g_{V}\in S^{2} (V^{*})$ be the constant extension of $g_{0}$
to $V$. It follows  that $(V,D , \Phi, B_{0}, B_{\infty}, g_{V})$
is a Saito bundle with metric (see Definition \ref{saito-def}).
The section $s$ is primitive homogeneous, with $B_{\infty}(s) =
(1-\frac{D}{2}) s$ (we use (\ref{b-infty}); recall that $s$ is the
constant section of $V$, determined by $v_{0} = (1,0,\cdots , 0)
\in \mathbb{C}^{m}$). The metric $g_{\mathrm{can}}(X, Y) := g_{V}(
I(X), I(Y))$ extends $(g_{\mathrm{can}})_{0}$ (from (\ref{g-0})).
From Subsection \ref{frob-saito}, $g_{\mathrm{can}}$ is a
Frobenius metric on $((M^{\mathrm{can}},0) ,\circ_{\mathrm{can}},
\mathrm{Id}_{\mathrm{can}}, E_{\mathrm{can}})$ and
$$
D^{\mathrm{LC}}E_{\mathrm{can}} = I^{-1}  B_{\infty} I
+\frac{D}{2} \mathrm{Id}.
$$
In order to conclude the proof, we need to check that
\begin{equation}\label{last-check}
I^{-1}_{0}  B_{\infty} I_{0} = (f_{*})_{p}\circ {\mathcal
V}_{p}\circ (f^{-1}_{*})_{p}.
\end{equation}
From (\ref{a-o-i}), the left hand side of (\ref{last-check}),
applied to $(B_{0}^{o})^{i}$ (with $0\leq i\leq n-1$), is given by
$$
(I^{-1}_{0}  B_{\infty} I_{0})(B_{0}^{o})^{i}=
I_{0}^{-1}B_{\infty} (v_{i})= (B_{\infty})_{ji} (B_{0}^{o})^{j}.
$$
From (\ref{ultima-relatie}), the right hand side of
(\ref{last-check}), applied to $(B^{o}_{0})^{i}$,  is given by
$$
((f_{*})_{p}\circ {\mathcal V}_{p}\circ (f^{-1}_{*})_{p})
(B_{0}^{o})^{i} = (f_{*})_{p} {\mathcal V}_{p} (E_{p}^{i}) =
(B_{\infty})_{ji} (f_{*})_{p} (E_{p}^{j}) = (B_{\infty})_{ji}
(B_{0}^{o})^{j}.
$$
Relation (\ref{last-check}) follows.
\end{proof}

{\it Liana David}:  Institute of Mathematics 'Simion Stoilow' of the
Romanian Academy, Research Unit 4, Calea Grivitei nr. 21,
Bucharest, Romania; liana.david@imar.ro\\

{\it Claus Hertling}: Lehrstuhl f\"{u}r Mathematik VI, Institut f\"{u}r Mathematik,
Universit\"{a}t  Mannheim,
A5, 6, 68131, Mannheim, Germany;
hertling@math.uni-mannheim.de

\end{document}